\documentclass[twoside]{amsart}

\usepackage{ifpdf}
\ifpdf
 \usepackage[implicit,naturalnames,hyperfootnotes=false,%						% Links
						 final]{hyperref}%																			% 
 \usepackage[pdftex]{color}%																				% Farben
\else%																															%
 \usepackage{color}																									% Farben
 \gdef\texorpdfstring#1#2{#1}%																			%
\fi%																																%
\usepackage{lmodern}
\usepackage{xcolor}
\usepackage{enumitem}
\usepackage{amsmath}
\usepackage{amssymb}
\usepackage{amsthm}
\usepackage{bbm}
\usepackage{amsfonts}
\usepackage{mathrsfs}
\usepackage{calligra}
\usepackage{dsfont}
\usepackage{nicefrac}
\usepackage{esint}
\usepackage{mathtools}
\usepackage{ifthen}
\usepackage{xparse}
\usepackage{xspace}
\usepackage{calc}
\usepackage[nostandardtensors]{simpletensors}
\usepackage{partial}
\usepackage{bracket-hack}
\usepackage{auto-eqlabel}
\NewDocumentCommand\term{so}{\IfValueTF{#2}\term\emph}
\newtheoremstyle{mytheorem}{3pt}{}{\itshape}{}{\bfseries}{\nopagebreak\newline}{.5em}{}%
\newtheoremstyle{mydefinition}{3pt}{}{}{}{\bfseries}{\nopagebreak\newline}{.5em}{}%
\gdef\mytheorem{\theoremstyle{mytheorem}}%													%
\gdef\mydefinition{\theoremstyle{mydefinition}}%										%
\gdef\mytheoremcounter{theorem}%																		%
\mytheorem%																												%
\newtheorem{theorem}{Theorem}[section]%														%
\newtheorem{lemma}[\mytheoremcounter]{Lemma}%											%
\newtheorem{proposition}[\mytheoremcounter]{Proposition}%					%
\newtheorem{assumptions}[\mytheoremcounter]{Assumptions}%					%
\mydefinition%																											%
\newtheorem{definition}[\mytheoremcounter]{Definition}%						%
\theoremstyle{remark}%																							%
\newtheorem{remark}[\mytheoremcounter]{Remark}%										%
\gdef\ii{I}\gdef\ij{J}\gdef\ik{K}\gdef\il{L}%												% 2D
\gdef\oi{i}\gdef\oj{j}\gdef\ok{k}\gdef\ol{l}%												% 3D
\NewDocumentCommand\outsymbol{om}{\overline{\IfValueTF{#1}{#1{#2}}{#2}}}
\NewDocumentCommand\unisymbol{om}{\widehat{\IfValueTF{#1}{#1{#2}}{#2}}}
\NewDocumentCommand\outtensor{d()}{\IndexSymbol(\outsymbol[#1])}%		%
\makeatletter%																											%
\newdimen\middle@width%																							%
\newcommand*\phantomas[3][c]{\ifmmode\makebox[\widthof{$#2$}][#1]{$#3$}\else\makebox[\widthof{#2}][#1]{#3}\fi}
\NewDocumentCommand\tracefree{m}%																		% Spurfreies Symbol #1
{\setlength\middle@width{\widthof{\ensuremath{#1}}}%								%
 #1\hskip-\middle@width%																						%
 \phantomas[r]{#1}{\vphantom{\ensuremath{#1}}^{\,\!^\circ}}}%				%
\NewDocumentCommand\mean{m}%																				% Spurfreies Symbol #1
{\setlength\middle@width{\widthof{\ensuremath{#1}}}%								%
 \phantomas{#1}{{=\!}}%																							%
 \hskip-\middle@width#1}%																						%
\makeatother%																												%
\NewTensor[\outsymbol]\mass m%																			%
\let\oldPhi\Phi\RenewTensor\Phi\oldPhi%															%
\NewTensor[\outsymbol]\outPhi\oldPhi%																%
\let\oldPsi\Psi\RenewTensor\Psi<\!>\oldPsi%													%
\NewTensor[\outsymbol]\outPsi\oldPsi%																%
\NewTensor[\outsymbol]\outx x%																			%
\NewTensor[\outsymbol]\outy y%																			%
\NewTensor[\outsymbol]\outp p%																			%
\NewDocumentCommand\intervalI{s}{\IfBooleanTF{#1}I{\ensuremath{\intervalI*}\xspace}}%
\NewDocumentCommand\intervalJ{s}{\IfBooleanTF{#1}J{\ensuremath{\intervalJ*}\xspace}}%
\gdef\schwarzs{\mathcal S}%											% Schwarzschild S
\gdef\schwarzoutg{\outg[\schwarzs]}%																% Schwarzschild Metrik
\gdef\euclideane{e}%														% Euklidisches e
\gdef\eukoutg{\outg[\euclideane]}%																	% Euklidsche Metrik
\gdef\eukouttr{\outtr[\euclideane\,]}%															% Euklidsche Spur
\gdef\eukoutdiv{\outdiv[\euclideane]}%															% Euklidsche Divergenz
\gdef\eukoutHess{\outHess[\euclideane]}%														% Euklidsche Hessische
\gdef\eukoutlaplace{\outlaplace[\euclideane]}%											% Euklidscher Laplace
\gdef\eukmug{\mug[\euclideane]}%																		% Euklidsches Oberflächenmaß
\gdef\eukzFundtrf{\zFundtrf[\euclideane\hspace{.05em}]}%						% Euklidsche Metrik
\gdef\rad{|\outx|}%																									% Koordinaten-Radius
\gdef\atime{\tau}%																									% Künstliche Zeit
\gdef\volume#1{\left|#1\right|}%																		%
\undef\c\NewDocumentCommand\c{sG{c}}{\IfBooleanTF{#1}{#2}{\IndexSymbol*#2}}% 2D-Konstanten
\gdef\cSob{\c{\c_S}}% Sobolev Konstanten
\NewDocumentCommand\Cof{G{C}d<>d()}%																% Abhängigkeit der Konstanten hervorheben (std.)
{\IfValueTF{#3}{\Cof{#1}<\IfValueTF{#2}{#2,}\relax{#3}>}{\CofSnd{#1\IfValueTF{#2}{[#2]}\relax}}}
\NewDocumentCommand\CofSnd{G{C}d<>o}%																% Abhängigkeit der Konstanten hervorheben (std.)
{\IfValueTF{#3}{\CofSnd{#1}<\IfValueTF{#2}{#2,}\relax{#3}>}{\mathop{#1\IfValueTF{#2}{(#2)}\relax}}}%
\NewTensor\M[\hspace{-.05em}]<\hspace{-.05em}>\Sigma%								% 2D-Mfkt.
\NewTensor[\outsymbol]\outM[\hspace{-.025em}]{\textrm M}%						% 3D-Mfkt.
\NewTensor[\unisymbol]\uniM[\hspace{-.025em}]{\textrm M}%						% 4D-Mfkt.
\NewTensor\genus g%																									% Geschlecht
\NewTensor*\graphf[\!]f%																						% Graph-Funktion
\NewTensor\graphF F%																								% Graph-Abbildung
\NewTensor\sphg{\Omega}%																						% Metrik der Euklidischen Sphäre
\NewTensor*\rnu[\!]u%																								% 2D-Lapse Funktion
\NewTensor*\tnu[\hspace{-.05em}] w%																	% 2D-Lapse Funktion
\NewTensor*\lapse[\!]u%																							% 2D-Lapse Funktion
\NewTensor[\outsymbol]*\ralpha[\hspace{-.05em}]\alpha%							% 3D-Lapse Funktion
\NewTensor\rbeta[\!]\beta%																					% 2D-Shift
\NewTensor[\outsymbol]\outbeta[\!]\beta%														% 3D-Shift
\NewTensor[\MakeSymbol<>\vec]*\centerz[\!]<\hspace{-.05em}>z%				% 2D-Zentrum
\NewTensor[\MakeSymbol<\outsymbol>\vec]*\outcenterz[\!]<\hspace{-.05em}>Z% 3D-Zentrum
\undef\metric\NewTensor[\newmathcal]\metric[\!\!]<\hspace{.05em}>g%	% 2D-Metrik
\let\g\metric%																											% 2D-Metrik
\NewTensor[\outsymbol[\newmathcal]]\outmetric<\hspace{.05em}>g%			% 3D-Metrik
\let\outg\outmetric%																								% 3D-Metrik
\NewTensor[\unisymbol[\newmathcal]]\unimetric[\!\!]<\hspace{.05em}>g% 4D-Metrik
\let\unig\unimetric%																								% 4D-Metrik
\NewTensor[\newmathcal]\riem R%																			% 2D-Riemannsche Krümmung
\NewTensor[\outsymbol[\newmathcal]]\outriem R%											% 3D-Riemannsche Krümmung
\let\outrc\outriem%																									% 3D-Riemannsche Krümmung
\NewTensor[\normalfont\textrm]\ricci{Ric}%													% 2D-Ricci Krümmung
\let\ric\ricci%																											% 2D-Ricci Krümmung
\let\Ric\ricci%																											% 2D-Ricci Krümmung
\NewTensor[\outsymbol[\normalfont\textrm]]\outricci{Ric}%						% 3D-Ricci Krümmung
\let\outric\outricci%																								% 3D-Ricci
\NewTensor[\unisymbol[\normalfont\textrm]]\uniricci{Ric}%						% 4D-Ricci Krümmung
\NewTensor[\newmathcal]\scalar S%																		% 2D-Skalar Krümmung
\NewTensor[\outsymbol[\newmathcal]]\outscalar S%										% 3D-Skalar Krümmung
\let\outsc\outscalar%																								% 3D-Skalar Krümmung
\NewTensor[\unisymbol[\newmathcal]]\uniscalar S%										% 4D-Skalar Krümmung
\NewTensor[\newmathcal]\sectional K%																% 2D-Schnittkrümmung
\NewTensor[\newmathcal]\gauss G%																		% 2D-Gauß-Krümmung
\NewTensor[\textrm]\II k%																						% 2D-Zweite Fundamentalform
\let\zFund\k%																												% 2D-Zweite Fundamentalform
\NewTensor[\tracefree]\IItrf[\!]{\textrm k}%												% 2D-Zweite Fundamentalform (spurfrei)
\let\zFundtrf\ktrf%																									% 2D-Zweite Fundamentalform (spurfrei)
\NewDocumentCommand\vzFundtrf{d<>od<>}
{{\IfDisplaystyleTF\left\relax|\zFundtrf<#1>[#2]<#3>\IfDisplaystyleTF\right\relax|}}
\NewTensor[\outsymbol[\textrm]]\outII[\hspace{-.05em}]k%						% 2D-Zweite Fundamentalform
\let\outzFund\outk%																									% 3D-Zweite Fundamentalform
\NewDocumentCommand\troutzFund{D<>{}O{}}{\tr<#1>[#2]\hspace{.05em}\outzFund[#2]}%	% 2D Spur der 3D-Zweiten Fundamentalform
\NewDocumentCommand\outzFundnu{D<>{}O{}}{\mathop{{\outzFund[#2]_{\nu<#1>[#2]}}}}%	3D-Zweiten Fundamentalform in Normalenrichtung
\NewDocumentCommand\outzFundnunu{D<>{}O{}}{\mathop{{\outzFund[#2]_{\nu<#1>[#2]\nu<#1>[#2]}}}}%	3D-Zweiten Fundamentalform in Normalenrichtung (doppelt)
\NewTensor[\mathcal]\mc[\!]<\hspace{-.05em}> H%											% 2D-Mittlere Krümmung
\let\H\mc%																													% 2D-Mittlere Krümmung
\NewTensor[\outsymbol[\mathcal]]\outmc[\!]<\hspace{-.05em}>H%				% 2D-Mittlere Krümmung
\let\oldnu\nu%																											% 2D-Normale
\NewTensor*\normal[\!]\oldnu%																				% 2D-Normale
\let\nu\normal%																											% 2D-Normale
\NewTensor[\outsymbol]\timenormal[\!]\vartheta%											% 2D-zeitliche Normale
\NewTensor*\laplace[\!]\Delta%																			% 2D-Laplace
\NewTensor[\outsymbol]*\outlaplace[\!]\Delta%												% 3D-Laplace
\NewTensor\Hess[\!]{\text{Hess}}%																		% 2D-Hessische
\NewTensor\Hesstrf[\!]{\text H\tracefree{\text{es}}\text s}%				% 2D-Hessische (spurfrei)
\NewTensor[\outsymbol]\outHess[\!]{\text{Hess}}%										% 3D-Hessische
\undef\div\NewTensor\div[\!]{\text{div}}%														% 2D-Divergenz
\NewTensor[\outsymbol]\outdiv[\!]{\text{div}}%											% 3D-Divergenz
\NewTensor[\unisymbol]\unidiv[\!]{\text{div}}%											% 4D-Divergenz
\NewTensor*\tr[\!]{\text{tr}}%																			% 2D-Spur
\NewTensor[\outsymbol]*\outtr{\text{tr}}%														% 3D-Spur
\NewTensor[\unisymbol]*\unitr[\!]{\text{tr}}%												% 4D-Spur
\NewTensor\levi<\MakeSymbol<>{\!}>{\MakeSymbol<\Gamma>\nabla}%			% 2D-Zusammenhang
\NewTensor[\outsymbol]\outlevi<\MakeSymbol<>{\!}>{\MakeSymbol<\Gamma>\nabla}% 3D-Zusammenhang
\NewTensor[\unisymbol]\unilevi<\MakeSymbol<>{\!}>{\MakeSymbol<\Gamma>\nabla}% 4D-Zusammenhang
\NewTensor*\vol{\text{vol}}%																				% 3D-Euklidsches Volumenmaß
\NewTensor*\mug[\!]\mu%																							% 2D-Oberflächenmaß
\NewTensor[\outsymbol]*\outmug[\!]\mu%															% 3D-Volumenmaß
\NewTensor[\outsymbol]\momentum[\!]\Pi%															% Momenten Tensor (k-Hg)
\NewTensor[\outsymbol]\einstein{{\normalfont\textrm G}}%						% Einsteintensor (entspricht uniric(*,*) + 1/2 unisc unig)
\NewTensor[\outsymbol]\outmomden[\!]J%															% Momentendichte (entspricht uniric(mu,*))
\NewTensor[\outsymbol]\outenden[\hspace{-.05em}]\varrho%						% Energiedichte (entspricht uniric(mu,mu))
\NewTensor[\textrm]*\jacobiext[\hspace{-.05em}]<\hspace{-.05em}> L%	% Jacobi-Operator
\NewDocumentCommand\jacobit{G{\pm}}{\IndexSymbol*[\hspace{-.05em}]<\hspace{-.05em}>{\textrm L_{#1}^{\!t}}}% Jacobi-Operator
\NewDocumentCommand\jacobipm{G{\gewicht}}{\IndexSymbol*[\hspace{-.05em}]<\hspace{-.05em}>{\textrm L_{#1}}}% Jacobi-Operator
\NewDocumentCommand\ajacobipm{G{\gewicht}}{\IndexSymbol*[\hspace{-.05em}]<\hspace{-.05em}>{\textrm L_{#1}^{\!\!a}}}% Jacobi-Operator
\NewDocumentCommand\trzd{smm}{{#2}\odot{#3}}%												% 2-3 Spur
\NewDocumentCommand\trtr{smm}%																			% 2-3 Spur
{\ifthenelse{\equal{#2}{#3}}{\left|#2\right|_{\g*}^2}{\tr(\trzd{#2}{#3})}}%
\NewDocumentCommand\outc{G{c}}{\IndexSymbol(\outsymbol)*{#1}}\let\oc\outc%			% 3D-Konstanten
\NewTensor\outdelta\delta%																					% 3D-Konstante delta
\NewTensor\outve\ve%																								% 3D-Konstante delta
\NewDocumentCommand\outck{d<>o}{\outc<#1>[#2]_{\outzFund*}\vphantom{\outc_{\outzFund*}}}%			% 3D-Konstanten
\gdef\outtr{\IndexSymbol(\outsymbol)[\!]{\text{tr}}}%								% 3D-Spur
\NewDocumentCommand\outtrzd{D<>{}O{}mm}%														% 2-3-Spurung
{\mathop{{#3\IndexSymbol(\outsymbol)[#2]<#1>\odot{#4}}}}%						%
\NewDocumentCommand\outtrtr{D<>{}O{}mm}%														%
{\mathop{{\ifthenelse{\equal{#3}{#4}}%															%
 {{\left|#3\right|_{\outg<#1>[#2]}^2}}%															%
 {\mathop{{\outtr<#1>[#2](\outtrzd<#1>[#2]{#3}{#4})}}}}}}%					%
\NewTensor*\eflap[\!]f%																							% Eigenfunktionen des Laplace
\NewTensor*\ewlap[\!]\lambda%																				% Eigenwerte des Laplace
\NewTensor*\efjac[\!]{\mathcal f}%																	% Eigenfunktionen des Stabilitätsoperators
\NewTensor*\ewjac[\!]<\!>\gamma%																		% Eigenwerte des Stabilitätsoperators
\NewTensor*\funcg[\!]g%																							% Einfache eine Funktion
\NewTensor*\funch[\!]h%																							% Einfache eine Funktion
\gdef\trans#1{{#1}^{\text t}}%																			%
\gdef\deform#1{{#1}^{\text d}}%																			%
\gdef\Hradius{\sigma}%																							% Radius nach mittlerer Krümmung
\gdef\Aradius{r}%																										% Radius nach Fläche
\gdef\rradius{R}%																										% Radius nach minimalen Abstand zum Flächen-Zentrum
\NewTensor\mHaw{{\normalfont m_{\text{H}}}}%												% Hawking-Masse
\NewTensor[\MakeSymbol<\outsymbol>\vec]\impuls P%										% Impuls
\undef\d%																														%
\NewDocumentCommand\d{s}{\IfBooleanTF{#1}\relax{\mathop{}\!}\mathrm d}%
\newcommand\ds{{\d s}}%																							% Std. Differentialform für Integrale
\DeclareMathOperator\id{id}%																				% Identität
\DeclareMathOperator\graph{graph}%																	% Graph
\newcommand\R{\mathds{R}}																						% reelle Zahlen
\newcommand\N{\mathds{N}}																						% natürliche Zahlen
\newcommand\X{\mathfrak{X}}																					% Menge aller Vektorfelder
\NewDocumentCommand\Lp{t^}{\IfBooleanTF{#1}{\Lpsnd{^}}{\Lpsnd\relax\relax}}%															% Lebesgue Räume
\newcommand\Lpsnd[2]{{\normalfont\textrm L#1{#2}}}
\newcommand\Wkp{{\normalfont\textrm W}}%														% Sobolev Räume
\newcommand\Hk{{\normalfont\textrm H}}%															% Sobolev Räume (p=2)
\newcommand\Ck{\mathcal C}%																					% Stetige Funktionen
\newcommand\sphere{\mathds S}																				% Sphäre
\newcommand\ve{\varepsilon}																					% schönes Epsilon
\usepackage{urwchancal}																							% Ersetzt mathcal
\DeclareMathAlphabet{\mathcal}{OT1}{pzc}{m}{n}%											%
\let\newmathcal\mathcal%																						%
\usepackage{auto-eqlabel}																						% Nummerierungen nur bei referenzierten Formeln

\title[Foliations by stable spheres with constant mean curvature]{Foliations by stable spheres with\\constant mean curvature for isolated\\systems without asymptotic symmetry}
 \author[Christopher Nerz]{Christopher Nerz}
 \address{Institutionen f\"or matematik \\Kungliga Tekniska h\"ogskolan \\Stockholm \\ Sverige}
 \email{ncroman@kth.se}
 \date\today

\begin{document}
\begin{abstract}
In 1996, Huisken-Yau showed that every three-dimensional Riemannian manifold can be uniquely foliated near infinity by stable closed CMC-surfaces if it is asymptotically equal to the (spatial) Schwarz\-schild solution and has positive mass. Their assumptions were later weakened by Metzger, Huang, Eichmair-Metzger and others. We further generalize these existence results in dimension three by proving that it is sufficient to assume asymptotic flatness and non-vanishing mass to conclude the existence and uniqueness of the CMC-foliation and explain why this seems to be the conceptually optimal result. Furthermore, we generalize the characterization of the corresponding coordinate CMC-center of mass by the ADM-center of mass proven previously by Corvino-Wu, Huang, Eichmair-Metzger and others (under other assumptions).
\end{abstract}
\maketitle
\let\sc\scalar%
\section*{Introduction}
In order to study the quasi-local mass of asymptotically flat manifolds, Christo\-doulou-Yau used surfaces of constant mean curvature (CMC) \cite{christodoulou71some}. Since then, CMC-surfaces have proven to be a useful tool for mathematical general relativity. It was first proven by Huisken-Yau in 1996 that every three-dimensional Riemannian manifold $(\outM,\outg*,\outx)$ can be uniquely foliated near infinity by closed CMC-surfaces if it is asymptotic to the (spatial) Schwarzschild solution \cite{huisken_yau_foliation}. Besides proving this existence and uniqueness result, they showed that this foliation can be used as a definition of the center of mass.
Here, being asymptotic to (spatial) Schwarzschild solution means that there exists a coordinate system $\outx:\outM\setminus\outsymbol L\to\R^3\setminus \overline{B_1(0)}$ mapping the manifold (outside some compact set $\outsymbol L$) to Euclidean space, such that the push-forward of the metric $\outg*$ is asymptotically equal to the Schwarzschild metric $\schwarzoutg*$ as $\vert\outx\vert\to\infty$. Huisken-Yau assumed that the $k$-th derivatives of the difference $\outg_{ij}-\schwarzoutg_{ij}$ of the metric $\outg*$ and the Schwarzschild metric $\schwarzoutg*$ decay like $\vert\outx\vert^{-2-k}$ in these coordinates for every $k\le4$. This is abbreviated by writing $\outg*-\schwarzoutg*=\mathcal O_4(\vert x\vert^ {-2})$.

Later, Metzger proved the same result, but weakened their decay assumptions to $\outg*-\schwarzoutg*=\mathcal O_2(\vert x\vert^ {-1-\outve})$ for $\ve\ge0$, i.\,e.\ he only had to assume decay of the difference between the metrics, the corresponding Christoffel symbols and the corresponding curvatures and additionally he reduced the assumed decay rate \cite{metzger2007foliations}.%
\footnote{Note that he allowed $\outve=0$ if the constants of the corresponding inequalities are sufficiently small. This is a very interesting, particular result as he does not assume that the scalar curvature is integrable.}
However, this means that he still assumed that the metric is rotationally symmetric up to order $\vert x\vert^{-1-\outve}$. This symmetry assumption was weakened by Huang who proved that it is sufficient that the metric is asymptotic to the Euclidean one (\term{asymptotically flat}) with $\outg*-\eukoutg*=\mathcal O_5(\vert x\vert^{-\frac12-\outve})$, the scalar curvature decays with $\outsc=\mathcal O_0(\vert x\vert^{-3-\outve})$, and the mass is not zero if additionally metric and scalar curvature are (asymptotically) invariant under reflection at the coordinate origin (\term*{Regge-Teitelboim condition} \cite{regge1974role}, see Definition \ref{Regge-Teitelboim_conditions}) \cite{Huang__Foliations_by_Stable_Spheres_with_Constant_Mean_Curvature}. Furthermore, the corresponding result was proven by Eichmair-Metzger in dimensions greater than three if the metric is asymptotic to the Schwarzschild metric \cite{metzger_eichmair_2012_unique}. \pagebreak[3]

Under her assumptions, Huang additionally proves that the CMC-center of mass coincides with the (ADM-)center of mass\vspace{-.5em}
\[ \frac1{16\pi m}\lim_{\rradius\to\infty}\int_{\sphere^2_\rradius(0)} \sum_{j=1}^3 (\outx^i(\partial[\outx]_\oj@{\outg_{\oj\ok}}-\partial[\outx]_\ok@{\outg_{\oj\oj}})\outx^\ok-(\outg_\oj^\oi\hspace{.05em}\frac{\outx^\oj}\rradius-\outg_{\oj\oj}\hspace{.05em}\frac{\outx^\oi}\rradius)) \d\eukmug \vspace{-.5em}\]
defined by Regge-Teitel\-boim \cite{regge1974role} and Beig-{\'O}\,Murchadha \cite{beig1987poincare}. We use the name \lq ADM-center of mass\rq\ as this definition is similar to Arnowitt-Deser-Misner's definitions of mass and linear momentum \cite{arnowitt1961coordinate}. The same result was previously proven by Corvino-Wu and later by Eichmair-Metzger (under different assumptions) \cite{Corvino__On_the_center_of_mass_of_isolated_systems,metzger_eichmair_2012_unique}. It was proven by Cederbaum and the author that this results does not hold if the Regge-Teitel\-boim conditions are not satisfied \cite{cederbaumnerz2013_examples}. \pagebreak[2]%

Note that the CMC-foliation is not the only foliation used in mathematical general relativity. For example, Metzger proved existence and uniqueness of a foliation by spheres of constant expansion \cite{metzger2007foliations}, Lamm-Metzger-Schulze proved a corresponding result for spheres of Willmore type \cite{lamm2011foliationsbywillmore}, and (in the static case) Cederbaum proved that the level-sets of the static lapse function form a unique foliation \cite{Cederbaum_newtonian_limit}.\pagebreak[3]\smallskip

In this paper, we generalize the above results for the CMC-foliation in dimension three by proving that it exists (Theorem \ref{Existence_of_a_CMC-foliation}) and is unique (Theorem \ref{Uniqueness_of_the_CMC_surfaces}) if the metric is asymptotically flat with asymptotically vanishing scalar curvature and non-vanishing mass $\mass$, i.\,e. $\outg*=\eukoutg*+\mathcal O_2(\vert x\vert^{{-}\frac12-\ve})$, $\outsc=\mathcal O(\vert x\vert^{{-}3-\outve})$, and $\mass\neq0$. To the best knowledge of the author, this is the first time that existence (and uniqueness) of the CMC-foliation could be proven without assuming any (asymptotic) symmetry condition on the metric (and the scalar curvature). It should be noted that these decay assumptions are the pointwise version of the (Sobolev-)decay assumptions ($\outg*-\eukoutg*\in\Wkp^{2,p}_{\nicefrac12}$ with $\outsc\in\Lp^1$) made by Bartnik in order to prove that the ADM-mass is well-defined \cite{bartnik1986mass} and that Bartnik's decay assumptions are optimal \cite{denisov1983energy}. It is therefore reasonable to presume that these decay assumptions cannot be weakened -- except by replacing the pointwise by the corresponding Sobolev decay assumptions. Note that the proof presented here also works under the decay assumption $\outg*-\eukoutg*\in\Wkp^{3,p}_{\nicefrac12}$ with $\outsc\in\Lp^1$ and $p>2$ (see Section~\ref{Weaker_decay_assumptions}).

Additionally, we prove that the leaves of this foliation are stable if and only if the mass $\mass$ is positive. More precisely, we prove that the first three eigenvalues of the stability operator of such a leaf are equal to $\frac34\,\mass\,|\H|^3$ (at highest order) (Theorem \ref{Stability_of_the_foliation}) and that the corresponding eigenfunction correlate to translations (Propositions \ref{Regularity_of_surfaces_in_asymptotically_flat_spaces} and \ref{Movement_of_the_spheres_by_the_lapse_function}), where $\H$ and $\mHaw$ denote the mean curvature and the Hawking mass of the leaf, correspondingly. Analogous results were also proven by Huisken-Yau, Metzger, Huang, and Eichmair-Metzger under the corresponding decay assumptions explained above \cite{huisken_yau_foliation,metzger2007foliations,Huang__Foliations_by_Stable_Spheres_with_Constant_Mean_Curvature,metzger_eichmair_2012_unique}.

Furthermore, we prove that the CMC-center of mass exists and is equal to the ADM-center of mass if we additionally assume that the Regge-Teitel\-boim conditions are satisfied (Theorem \ref{The_centers_of_mass}). More precisly, we prove that this equality also holds if we only assume a weaker form of the Regge-Teitel\-boim conditions. However, as the CMC-center of mass does not need to be well-defined under these assumptions \cite{cederbaumnerz2013_examples}, the latter is true in the sense that the CMC-center of mass is well-defined if and only if the ADM-center of mass is well-defined, and in that case these centers coincide. This generalizes the results cited above (in dimension three) and \cite[Cor.~5.3]{nerz2013timeevolutionofCMC}.\pagebreak[3]\smallskip

\textbf{Acknowledgment.}
The author wishes to express gratitude to Gerhard Huisken for suggesting this topic, many inspiring discussions and ongoing supervision. Further thanks is owed to Lan-Hsuan Huang for suggesting the use of the Bochner-Lichnerowicz formula in this setting  -- a central step in the argument (Lemma \ref{Eigenvalues_of_the_Laplacian}). Finally, this paper would not have attained its current form and clarity without the useful suggestions by Carla Cederbaum.\pagebreak[3]

\section*{Structure of the paper}
In Section \ref{Assumptions_and_notation}, we explain the basic notations and definitions. We give the main regularity arguments in Section \ref{Regularity_of_the_hypersurfaces}. Existence and uniqueness of the CMC-foliation as well as the stability of its leaves are proven in Section \ref{existence_of_the_CMC-foliation}. Finally in Section \ref{the_centers_of_mass}, we give the definitions of ADM- and CMC-center of mass and prove that these are equivalent under a weak form of the Regge-Teitel\-boim conditions.

\section{Assumptions and notation}\label{Assumptions_and_notation}
In order to study foliations (near infinity) of three-dimensional Riemannian manifolds by two-dimensional spheres, we will have to deal with different manifolds (of different or the same dimension) and different metrics on these manifolds, simultaneously. To distinguish between them, all three-dimensional quantities like the surrounding manifold $(\outM,\outg*)$, its Ricci and scalar curvature $\outric$ and $\outsc$ and all other derived quantities carry a bar, while all two-dimensional quantities like the CMC leaf $(\M,\g*)$, its second fundamental form $\zFund*$, the trace-free part of its second fundamental form $\zFundtrf:=\zFund*-\frac12\,(\tr\zFund*)\g*$, its Ricci, scalar, and mean curvature $\ric$, $\sc$, and $\H:=\tr\zFund$, its outer unit normal $\nu$, and all other derived quantities do not.

Here, we interpret the second fundamental form and the normal vector of a hypersurface as quantities of the surface (and thus as two-dimensional). For example, if $\M<\Hradius>$ is a hypersurface in $\outM$, then $\nu<\Hradius>$ denotes its normal (and \emph{not} $\IndexSymbol{\outsymbol\nu}<\Hradius>$). The same is true for the \lq lapse function\rq\ and the \lq shift vector\rq\ of a hypersurfaces arising as a leaf of a given deformation or foliation. Furthermore, we stress that the sign convention used for the second fundamental form results in a \emph{negative} mean curvature of the Euclidean coordinate spheres.%

If different two-dimensional manifolds or metrics are involved, then the lower left index will always denote the mean curvature index $\Hradius$ of the current leaf $\M<\Hradius>$, i.\,e.\ the leaf with mean curvature $\H<\Hradius>\equiv\nicefrac{{-}2}\Hradius$. Furthermore, quantities carry the upper left index $\euclideane$ and $\sphg*$ if they are calculated with respect to the Euclidean metric $\eukoutg*$ and the standard metric $\sphg<\Hradius>$ of the Euclidean sphere $\sphere^2_\Hradius(0)$, correspondingly. We abuse notation and suppress the index $\Hradius$, whenever it is clear from the context which metric we refer to.%

Finally, we use upper case latin indices $\ii$, $\ij$, $\ik$, and $\il$ for the two-dimensional range $\lbrace2,3\rbrace$ and lower case latin indices $\oi$, $\oj$, $\ok$, and $\ol$ for the three-dimensional range $\lbrace 1,2,3\rbrace$. The Einstein summation convention is used accordingly.
\pagebreak[3]\smallskip

As there are different definitions of \lq asymptotically flat\rq\ in the literature, we now give the decay assumptions used in this paper. 
\begin{definition}[\texorpdfstring{$\Ck^2_{\frac12+\ve}$}{C-k}-asymptotically flat Riemannian manifolds]\label{Ck_asymptotically_flat}
Let $\outve\relax>0$. A triple $(\outM,\outg,\outx)$ is called \term{$\Ck^2_{\frac12+\ve}$-asymptotically flat} Riemannian manifold if $(\outM,\outg)$ is a smooth manifold and $\outx:\outM\setminus\overline L\to\R^3$ is a smooth chart of $\outM$ outside a compact set $\overline L\subseteq\outM$ such that there exists a constant $\oc\ge0$ with\vspace{-.25em}
\begin{equation*} 
 \vert\outg_{ij}-\eukoutg_{ij}\vert + \rad\,\vert\outlevi_{ij}^k \vert + \rad^2\,\vert\outric_{ij} \vert + \rad^{\frac52}\,\vert\outsc\vert \le \frac{\oc}{\rad^{\frac12+\outve}},\vspace{-.25em}
 \labeleq{Decay_assumptions_g}\end{equation*}
where  $\eukoutg*$ denotes the Euclidean metric.\pagebreak[1] Here, these quantities are identified with their push-forward along $\outx$.
Arnowitt-Deser-Misner defined the \term{{\normalfont(\hspace{-.05em}}ADM-{\normalfont\hspace{.05em})}mass} of a $\Ck^2_{\frac12+\outve}$-asymptotically Riemannian manifold $(\outM,\outg*,\outx)$ by
\begin{equation*} \mass_{\text{ADM}} := \lim_{\rradius\to\infty} \frac1{16\pi}\sum_{\oj=1}^3\int_{\sphere^2_\rradius(0)}(\partial[\outx]_\oj@{\g_{\oi\oj}}-\partial[\outx]_\oi@{\g_{\oj\oj}})\,\nu^\oi \d\mug  \label{Definition_of_mass_ADM}, \end{equation*}
where $\nu$ and $\mug$ denote the outer unit normal and the area measure of $\M\hookrightarrow(\outM,\outg*)$ \cite{arnowitt1961coordinate}.
\end{definition}%
In the literature, the ADM-mass is also characterized using the curvature of $\outg*$:
\begin{equation} \mass := \lim_{\rradius\to\infty} \frac{{-}\rradius}{8\pi} \int_{\sphere^2_\rradius(0)} \outric*(\nu<\rradius>,\nu<\rradius>) - \frac\outsc2 \d\mug<\rradius>, \labeleq{Definition_of_mass}\end{equation}%
see the articles of Ashtekar-Hansen, Chru\'sciel, and Schoen \cite{ashtekar1978unified,schoen1988existence,chrusciel1986remark}.
Miao-Tam recently gave a proof of the characterization $m_{\text{ADM}}=\mass$ in the setting used within this paper, i.\,e.\ for any $\Ck^2_{\frac12+\outve}$-asymptotically flat manifold \cite{miao2013evaluation}.\footnote{The author thank Carla Cederbaum for bringing his attention to Miao-Tam's article \cite{miao2013evaluation}.} We recall further aspects of this total mass in Appendix~\ref{ricci_integrals_and_the_mass}.
\begin{remark}[Alternative decay assumptions]
In Section~\ref{Weaker_decay_assumptions}, we discuss slightly different assumptions than the one stated above. In particular, we can replace the pointwise assumption by Sobolev assumptions and can alter the assumptions on $\outsc$.
\end{remark}
Let us recall the Hawking mass \cite{hawking2003gravitational}.
\begin{definition}[Hawking mass]
Let $(\outM,\outg*)$ be a three-dimensional Riemannian manifold. For any closed hypersurface $\M\hookrightarrow\outM$ the \term{Hawking mass} is defined by
\[ \mHaw(\M) := \sqrt{\frac{\volume{\M}}{16\pi}}(1-\frac1{16\pi}\int\H^2\d\mug), \]
where $\H$ and $\mug$ denote the mean curvature and measure induced on $\M$, respectively.
\end{definition}%
It is well-known that
\begin{equation*} \mass = \lim_{\rradius\to\infty} \mHaw(\sphere^2_\rradius(0)) \tag{\ref{Definition_of_mass}'}, \end{equation*}
see for example \cite{schoen1988existence}. We recall and explain this in Appendix \ref{ricci_integrals_and_the_mass} in more detail.\pagebreak[3]\smallskip

As mentioned, we frequently use foliations. We will in the following characterize them infinitesimally by their lapse functions and their shift vectors.
\begin{definition}[Lapse functions, shift vectors]
Let $\theta>0$ and $\Hradius_0\in\R$ be constants, $I\supseteq\interval{\Hradius_0-\theta\Hradius}{\Hradius_0+\theta\Hradius}$ be an interval, and $(\outM,\outg*)$ be a Riemannian manifold. A smooth map $\Phi:I\times\M\to\outM$ is called \term{deformation} of the closed hypersurface $\M=\M<\Hradius_0>=\Phi(\Hradius_0,\M)\subseteq\outM$ if $\Phi<\Hradius>(\cdot):=\Phi(\Hradius,\cdot)$ is a diffeomorphism onto its image $\M<\Hradius>:=\Phi<\Hradius>(\M)$ and $\Phi<\Hradius_0>\equiv\id_{\M}$. The decomposition of $\spartial*_\Hradius\Phi$ into its normal and tangential parts can be written as
\[ \partial*_\Hradius\Phi = {\rnu<\Hradius>}\,{\nu<\Hradius>} + {\rbeta<\Hradius>}, \]
where $\nu<\Hradius>$ is the outer unit normal to $\M<\Hradius>$. The function $\rnu<\Hradius>:\M<\Hradius>\to\R$ is called \term{lapse function} and the vector field $\rbeta<\Hradius>\in\X(\M<\Hradius>)$ is called \term{shift} of $\Phi$. If $\Phi$ is a diffeomorphism (resp.\ diffeomorphism onto its image), then it is called a \term{foliation} (resp.\ a \term*[foliation!local]{local foliation}).\pagebreak[3]
\end{definition}
For notation convenience, we use the following abbreviated form for the contraction of two tensor fields.
\begin{definition}[Tensor contraction]
Let $(\M,\g*)$ be a Riemannian manifold. The \term{traced tensor product} of a $(0,k)$ tensor field $S$ and a $(0,l)$ tensor field $T$ on $(\M,\g*)$ with $k,l>0$ is defined by
\[ (\trzd ST)_{I_1\dots I_{k-1}\!J_1\dots J_{l-1}} := S_{I_1\dots I_{k-1} K}\,T_{LJ_1\dots J_{l-1}}\,\g^{KL}. \]
This definition is independent of the chosen coordinates. Furthermore, $\trzd{\trzd ST}U$ is well-defined if $T$ is a $(0,k)$ tensor field with $k\ge2$, i.\,e.\ $\trzd{(\trzd ST)}U=\trzd S{(\trzd TU)}$ for such a $T$.
\end{definition}
Finally, we specify the definitions of Lebesgue and Sobolev norms on compact Riemannian manifolds which we will use throughout this article.
\begin{definition}[Lesbesgue and Sobolev norms]
If $(\M,\g*)$ is a closed Riemannian manifold, then the \term{Lebesgue norms} are  defined by
\[ \Vert T\Vert_{\Lp^p(\M)} := (\int_{\M} \vert T\vert_{\g*}^p \d\mug)^{\frac1p}\quad\forall\,p\in\interval*1\infty, \qquad \Vert T\Vert_{\Lp^\infty(\M)} := \mathop{\text{ess\,sup}}\limits_{\M} \vert T\vert_{\g*}, \]
where $T$ is any measurable function (or tensor field) on $\M$.\pagebreak[2] Correspondingly, $\Lp^p(\M)$ is defined to be the set of all measurable functions (or tensor fields) on $\M$ for which the $\Lp^p$-norm is finite. If $\Aradius:=(\nicefrac{\volume{\M}}{\omega_n})^{\nicefrac1n}$ denotes the \term{area radius} of $\M$, where $n$ is the dimension of $\M$ and $\omega_n$ denotes the Euclidean surface area of the $n$-dimensional unit sphere, then the \term{Sobolev norms} are defined by
\[ \Vert T\Vert_{\Wkp^{k+1,p}(\M)} := \Vert T\Vert_{\Lp^p(\M)} + \Aradius\,\Vert\levi*T\Vert_{\Wkp^{k,p}(\M)}, \qquad \Vert T\Vert_{\Wkp^{0,p}(\M)} := \Vert T\Vert_{\Lp^p(\M)}, \]
where $k\in\N_{\ge0}$, $p\in\interval*1*\infty$ and $T$ is any measurable function (or tensor field) on $\M$ for which the $k$-th (weak) derivative exists. Correspondingly, $\Wkp^{k,p}(\M)$ is the set of all such functions (or tensors fields) for which the $\Wkp^{k,p}(\M)$-norm is finite. Furthermore, $\Hk^k(\M)$ denotes $\Wkp^{k,2}(\M)$ for any $k\ge1$ and $\Hk(\M):=\Hk^1(\M)$.
\end{definition}
\section{Regularity of the hypersurfaces}\label{Regularity_of_the_hypersurfaces}
In this section, we prove the main regularity results for the hypersurfaces which we study in the following sections. The following bootstrap argument for surfaces with small trace-free part of the second fundamental form is central for the following argument. Note that Metzger used a similar approach to conclude this decay rate of the trace-free part of the second fundamental form \cite[Prop.~3.3]{metzger2007foliations}.
\begin{proposition}[Bootstrap for the second fundamental form]\label{Bootstrap_for_trace_free_second_fundamental_form}
Let $(\M,\g*)$ be a closed hypersurface of a three-dimensional Riemannian manifold $(\outM,\outg*)$. Assume that there are constants $\kappa>1$, $\c_1$, $\c_2$, $\eta>0$, $p>2$, $\Aradius=\sqrt{\nicefrac{\volume{\M}}{4\pi}}$ and $\mean\H=\mean\H(\M)\in\R$ such that
\begin{align*}
 \Vert \vert\outric*(\nu,\cdot)\vert_{\g*} \Vert_{\Lp^p(\M)} \le{}& \frac{\c_1}{\Aradius^{\kappa+1-\frac2p}}, &
 \Vert \vert\outric*\vert_{\g*} \Vert_{\Lp^p(\M)} \le{}& \frac{\c_2}{\Aradius^{2+\eta-\frac2p}},  \\
 \Vert \H - \mean\H \Vert_{\Wkp^{1,p}(\M)} \le{}& \frac{\c_1}{\Aradius^{\kappa-\frac2p}}, &
 \vert \mean\H + \frac2\Aradius \vert \le{}& \frac{\c_2}{\Aradius^{1+\eta}},
\end{align*}
where $\nu$ is a unit normal of $\M\hookrightarrow(\outM,\outg*)$ and assume furthermore that the first Sobolev inequality holds on $\M$, i.\,e.\ there is a constant $\cSob<\infty$ such that
\[ \Vert f\Vert_{\Lp^2(\M)} \le \frac\cSob\Aradius\,\Vert f\Vert_{\Wkp^{1,1}(\M)} \qquad\forall\,f\in\Wkp^{1,1}(\M). \]
Then there are constants $\Aradius_1=\Cof{\Aradius_1}[\kappa][\eta][\c_1][\c_2][\cSob][p]$ and $C=\Cof[\kappa][\eta][\c_2][\cSob][p]$ such that the implication
\[ \Vert\hspace{.05em}\zFundtrf*\Vert_{\Lp^2(\M)} \le \frac2{9\,\cSob} \qquad\Longrightarrow\qquad
	\Vert\hspace{.05em}\zFundtrf\Vert_{\Lp^\infty(\M)} + \Aradius^{-1}\,\Vert\hspace{.05em}\zFundtrf*\Vert_{\Hk(\M)}\le \frac{\c_1\,C}{\Aradius^\kappa} \]
holds if $\Aradius\ge\Aradius_1$.
\end{proposition}
Note that we do not assume that $(\M,\g*)$ is a hypersurface in an asymptotically flat Riemannian manifold, but only assume smallness of the three-dimensional Ricci curvature on $\M$.
\begin{proof}
Let us recall the Simons-identity for the Laplacian of the second fundamental form \cite{simons1968minimal,schoen1975curvature}
\begin{align*}\labeleq{Simons-identity}
 \laplace\zFund*
	={}& \Hess \H - \levi*\outric_{\nu} + \div_2\outrc_{\cdot\cdot\cdot\nu} + \frac{\H^2}2\zFundtrf + \H\,\trzd\zFundtrf\zFundtrf - \trtr\zFundtrf\zFundtrf\,\zFund* \\
		& + \trzd{(\tr_{23}\outrc)}\zFundtrf - \outrc_{\cdot I\!J\cdot}\,\zFundtrf^{I\!J},
\end{align*}
where $(\tr_{23}\outrc)_{I\!J}:=\outrc_{I\!K}^K_{\!J}$ denotes the trace of the three-dimensional Riemannian curvature with respect to the second and third component and $\div_2\outrc_{\cdot\cdot\cdot\nu}$ denotes the divergence of $\outrc*(\cdot,\cdot,\cdot,\nu)$ with respect to the second component -- both calculated with respect to the metric $\g*$.
By integration by parts of $\int\trtr{\laplace\zFundtrf}\zFundtrf\d\mug$, this implies
\[
 \Vert\levi*\zFundtrf\Vert_{\Lp^2(\M)}^2\hspace{-.05em}
	\le\hspace{-.05em} \frac{\c_1\hspace{-.05em}C}{\Aradius^\kappa}\Vert\levi*\zFundtrf\Vert_{\Lp^2(\M)} \hspace{-.05em}-\hspace{-.1em} \int\frac{\H^2}2\trtr\zFundtrf\zFundtrf + \H\hspace{.05em}\tr(\trzd{\trzd\zFundtrf\zFundtrf}\zFundtrf) - \vert\hspace{.05em}\zFundtrf\vert_{\g*}^4\d\mug + \frac{\c_1}{\Aradius^\kappa}\Vert\hspace{.05em}\zFundtrf\Vert_{\Lp^4(\M)}^2\hspace{-.025em},
\]
where we used that $\dim\outM=3$ implies that the Riemannian curvature of $\outM$ is given by combinations of the Ricci curvature. Using the assumptions on $\H$, this implies
\begin{align*}
 \Vert\levi*\zFundtrf\Vert_{\Lp^2(\M)}^2
	\le{}& \frac{\c_1^2\,C^2}{\Aradius^{2\kappa}} - \frac{\mean\H^2}2\Vert\hspace{.05em}\zFundtrf\Vert_{\Lp^2(\M)}^2 + \frac12\Vert\H^2-\mean\H^2\Vert_{\Lp^2(\M)}\,\Vert\hspace{.05em}\zFundtrf\Vert_{\Lp^4(\M)}^2 \\
		& + \vert\mean\H\vert \Vert\hspace{.05em}\zFundtrf\Vert_{\Lp^4(\M)}^2\,\Vert\hspace{.05em}\zFundtrf\Vert_{\Lp^2(\M)} + \Vert\hspace{.05em}\zFundtrf\Vert_{\Lp^4(\M)}^3\,\Vert \H-\mean\H\Vert_{\Lp^4(\M)} \\
		& + (1+\delta)\,\Vert\hspace{.05em}\zFundtrf\Vert_{\Lp^4(\M)}^4
\end{align*}
for every $\delta>0$ \pagebreak[1]and a constant $C$ additionally depending on $\delta$ (and the above constants $\kappa$, $\eta$, $\c_1$, $\c_2$, $\cSob$, and $p$). As a simple calculation proves that the validity of the first Sobolev inequality implies that the other Sobolev inequalities also hold, i.\,e.
\[ \Vert f\Vert_{\Lp^{\frac{2p}{2-p}}(\M)} \le \frac{\cSob\,p}{(2-p)\,\Aradius}\,\Vert f\Vert_{\Wkp^{1,p}(\M)} \qquad\forall\,f\in\Wkp^{1,p}(\M),\;p\in\interval*12, \]
we see that the assumptions imply
\[ \frac3{2\Aradius^2} \le \mean\H^2 + \frac C {\Aradius^{2\kappa}}, \qquad
	 \Vert \H^2-\mean\H^2 \Vert_{\Lp^2(\M)} \le \frac{\c_1\,C}{\Aradius^\kappa}, \qquad
	 \Vert \H - \mean\H \Vert_{\Lp^4(\M)} \le \frac{\c_1\,C}{\Aradius^{\kappa-\frac12}}. \]
Thus, we conclude that for sufficiently large $\Aradius$
\[ \Vert\levi*\zFundtrf\Vert_{\Lp^2(\M)}^2 \le \frac{\c_1^2\,C^2}{\Aradius^{2\kappa}} - \frac1{4\Aradius^2}\Vert\hspace{.05em}\zFundtrf\Vert_{\Lp^2(\M)}^2 + \frac52\Vert\hspace{.05em}\zFundtrf\Vert_{\Lp^4(\M)}^4 \]
due to $\kappa>1$. By the Sobolev inequality, we know
\[ \Vert\hspace{.05em}\zFundtrf\Vert_{\Lp^4(\M)}^2 = \Vert \trtr\zFundtrf\zFundtrf\Vert_{\Lp^2(\M)} \le \c_S(2\Vert\levi*\zFundtrf\Vert_{\Lp^2(\M)} + \frac1\Aradius\Vert\hspace{.05em}\zFundtrf\Vert_{\Lp^2(\M)})\,\Vert\hspace{.05em}\zFundtrf\Vert_{\Lp^2(\M)},  \]
where we used that $\vert\levi*\vert\hspace{.05em}\zFundtrf\vert^2\vert = 2 \vert\hspace{.05em}\zFundtrf\vert\,\vert\levi*\zFundtrf\vert$ $\mug$-almost everywhere. Combining the last two inequalities, we get the claimed $\Hk(\M)$-inequality.\pagebreak[1] To conclude the $\Lp^\infty$-decay, we see for $q$ with $p^{-1}+q^{-1}=1$ and any $(0,2)$-tensor field $T\in\X(\M)$ with $\Vert T\Vert_{\Wkp^{1,q}(\M)}\le1$ the Simons identity implies
\[ \vert\int \trtr T{\laplace\zFundtrf} \d\mug \vert \le \frac{\c_1\,C}{\Aradius^{\kappa+2-\frac2p}} + \frac C{\Aradius^{3-\frac2p}}(\Vert\hspace{.05em}\zFundtrf\Vert_{\Hk(\M)}+\Vert\hspace{.05em}\zFundtrf\Vert_{\Hk(\M)}^2+\Vert\hspace{.05em}\zFundtrf\Vert_{\Hk(\M)}^3).\pagebreak[2] \]
Using the above inequality for $\Vert\hspace{.05em}\zFundtrf\Vert_{\Hk(\M)}$, we conclude
\[ \Vert\laplace\zFundtrf\Vert_{\Wkp^{-1,q}(\M)} \le \frac{\c_1\,C}{\Aradius^{\kappa+2-\frac2p}}, \]
where $\Wkp^{-1,q}(\M)$ denotes the dual space of $\Wkp^{1,q}(\M)$. Therefore, the regularity of the weak Laplacian implies that the $\Lp^\infty(\M)$-inequality holds.\pagebreak[3]
\end{proof}
We will use the following well-known lemma to show that the assumption of the \lq not too large trace-free part of the second fundamental form\rq\ can be expressed as a compatibility condition on the mean curvature and the area of the surface.
\begin{lemma}[Compatibility of area and mean curvature]\label{Compatibility_of_area_and_mean_curvature}
Let $(\M,\g*)$ be a closed hypersurface with genus $0$ of a three-dimensional Riemannian manifold $(\outM,\outg*)$ and let $\Aradius=\sqrt{\nicefrac{\volume{\M}}{4\pi}}$ denote its {\normalfont(\hspace{-.05em}}area-{\normalfont)}radius. Assume that there are constants $\kappa>1$ and $\c_1$ such that such that
\[ \Vert\outsc\Vert_{\Lp^1(\M)} \le \frac{\c_1}{\Aradius^{\kappa-1}}, \ \
	\Vert\outric*(\nu,\nu)\Vert_{\Lp^1(\M)} \le \frac{\c_1}{\Aradius^{\kappa-1}}, \ \
	\int_{\M}\H^2 \d\mug - 16\pi\,(1-\genus) \le \frac{\c_1}{\Aradius^{\kappa-1}} \]
where $\nu$ is a unit normal of $\M\hookrightarrow(\outM,\outg*)$. There is a constant $C=\Cof[\c_1][\kappa]$ such that
\[ \Vert\hspace{.05em}\zFundtrf\Vert_{\Lp^2(\M)} \le \frac{\c_1\,C}{\Aradius^{\frac{\kappa-1}2}}. \]
\end{lemma}
\begin{proof}
This is a direct corollary of the Gau\ss-Bonnet theorem and the Gau\ss\ equation.
\end{proof}
We combine Proposition \ref{Bootstrap_for_trace_free_second_fundamental_form}, Lemma \ref{Compatibility_of_area_and_mean_curvature} and DeLellis-M\"uller's result \cite[Thm~1.1]{DeLellisMueller_OptimalRigidityEstimates} to conclude better decay rates using asymptotically flatness of the surrounding manifold $(\outM,\outg*)$. This will allow us to prove uniqueness of the CMC-sur\-faces within the following class of surfaces.
\begin{definition}[\texorpdfstring{$\mathcal C_\eta(\c_0)$-}\relax asymptotically concentric surfaces\texorpdfstring{ ($\mathcal A^{\ve,\eta}_\Aradius(\c_0,\c_1)$)}\relax]\label{Not_of-center}
Let $(\outM,\outg*,\outx)$ be a $\Ck^2_{\frac12+\ve}$-asymptotically flat three-dimensional Riemannian manifold and $\eta\in\interval0*1$, $\c_0\in\interval*01$, and $\c_1\ge0$ be constants. A closed, oriented hypersurface $(\M,\g*)\hookrightarrow(\outM,\outg*)$ of genus $\genus$ is called \emph{$\mathcal C_\eta(\c_0)$-asymptotically concentric with radius $\Aradius=\sqrt{\nicefrac{\volume{\M}}{4\pi}}$ {\normalfont(}with constant $\c_1${\normalfont)}}, in symbols $\M\in\mathcal A^{\ve,\eta}_\Aradius(\c_0,\c_1)$ if
\begin{equation*}
 \vert\centerz\,\vert \le \c_0\,\Aradius+\c_1\,\Aradius^{1-\eta},\quad\
 \Aradius^{4+\eta} \le \min_{\M}\rad^{5+2\outve},\quad\
 \int_{\M}\H^2 \d\mug - 16\pi\,(1-\genus) \le \frac{\c_1}{\Aradius^\eta}, \labeleq{Not_of-center_ass}
\end{equation*}
where $\centerz=(\centerz_i)_{i=1}^3\in\R^3$ denotes the \emph{Euclidean coordinate center of $\M$} defined by
\[ \centerz_i := \fint_{\M}\outx_i\d\eukmug, \]
where $\eukmug$ denotes the measure induced on $\M$ by the Euclidean metric $\outx^*\eukoutg*$.
\end{definition}
Let us briefly explain the assumptions made here. The first two assumptions in \eqref{Not_of-center_ass} ensure that the surface $\M$ separates the inner part $\overline L$ of the surrounding manifold from infinity. This is necessary as Brendle-Eichmair proved that the CMC-surfaces are \emph{not} unique without this assumption (at least if we do not assume the scalar curvature $\outsc$ to be non-negative) \cite{brendle2013large}. The third assumption in \eqref{Not_of-center_ass} ensures that the radius $\Hradius$ of any CMC-surface defined by the mean curvature $\Hradius:=\nicefrac{{-}\H}2$ is comparable to the radius defined by the area $\Aradius:=\sqrt{\nicefrac{\volume{\M}}{4\pi}}$.\footnote{Note that by looking at the definition of the Hawking mass, we see that the third assumption in \eqref{Not_of-center_ass} can also be interpreted as an assumption on mass and genus of the surface.}  These assumptions are well-established in this setting, see for example \cite[Rem.~4.1,~4.(A3)]{metzger2007foliations} or (implicitly) \cite[Thm 4.1]{Huang__Foliations_by_Stable_Spheres_with_Constant_Mean_Curvature}.

\begin{proposition}[Regularity of surfaces in asymptotically flat manifolds]\label{Regularity_of_surfaces_in_asymptotically_flat_spaces}
Let $(\M,\g*)$ be a closed, oriented hypersurface in a $\Ck^2_{\frac12+\ve}$-asymptotically flat three-dim\-en\-sio\-nal Riemannian manifold $(\outM,\outg*,\outx)$ and let $\eta\in\interval0*1$, $\c_0\in\interval*01$, $\c_1\ge0$, and $p\in\interval2\infty$ be constants. If $\M\in\mathcal A^{\ve,\eta}_\Aradius(\c_0,\c_1)$ and if
\[
 \exists\,\mean\H(\M)\in\R:\qquad
 \Vert \H - \mean\H(\M) \Vert_{\Wkp^{1,p}(\M)} \le \frac{\c_1}{\Aradius^{\frac32+\ve-\frac2p}}, \]
then there are constants $\Aradius_1=\Cof{\Aradius_1}[\outve][\oc][\c_0][\c_1][\eta][p]$ and $C=\Cof[\outve][\oc][\c_0][\c_1][\eta][p]$ such that $\M$ is a sphere and
\begin{equation*} \Aradius^{-1}\,\Vert\hspace{.05em}\zFundtrf\Vert_{\Hk(\M)} + \Vert\hspace{.05em}\zFundtrf\Vert_{\Lp^\infty(\M)} \le \frac C{\Aradius^{\frac32+\outve}} \labeleq{Regularity_of_surfaces_in_asymptotically_flat_spaces__ineq_k} \end{equation*}
if $\Aradius>\Aradius_1$.\footnote{Note that we also get $\Vert\zFundtrf\Vert_{\Wkp^{1,p}(\M)}\le\nicefrac C{\Aradius^{\frac32+\ve-\frac2p}}$ by the same arguments using the weak regularity of the Laplace operator, e.\,g.\ a combination of \cite[Thm~7.1]{giaquinta2005introduction} and \cite[Thm~3.9]{adams2003sobolev} on the Simon's identity \eqref{Simons-identity} or \cite[Cor.~2.3.1.2]{christodoulou1993global} on $\levi*\laplace\zFundtrf$ using the Simon's identity \eqref{Simons-identity}.} In particular, \cite[Thm~1.1]{DeLellisMueller_OptimalRigidityEstimates} implies that there is a center point $\centerz\in\R^3$ and a function $f\in\Ck^2(\sphere^2_\Aradius(\centerz\,))$ such that\footnote{If $\vert\partial*_k\partial*_l\outg_{ij}\vert\le \nicefrac\oc{\rad^{\nicefrac52\,+\outve}}$ then we can use the Simons identity and the regularity of the weak Laplace operator, to conclude $\Vert f\Vert_{\Wkp^{3,p}(\M)}\le C\,\Aradius^{\nicefrac12\,+\nicefrac2p\,-\outve}$.}
\begin{equation*}
	\M = \graph f, \qquad
	\Vert f\Vert_{\Wkp^{2,\infty}(\sphere^2_\Aradius(\centerz\,))} \le C\,\Aradius^{\frac12-\outve},\qquad
	\vert\centerz\,\vert\le\c_0\Aradius + C\,\Aradius^{1-\eta}.
		\labeleq{Regularity_of_surfaces_in_asymptotically_flat_spaces__ineq_f} \end{equation*}
\end{proposition}
\begin{proof}
By assumption, we can use Lemma \ref{Compatibility_of_area_and_mean_curvature} and conclude $\Vert\hspace{.05em}\zFundtrf\Vert_{\Lp^2(\M)} \le \nicefrac C{\Aradius^\delta}$ for some $\delta>0$ depending only on $\eta$ and $\ve$. As the assumptions also imply $\vert\spartial[\outx]_k@{\outg_{ij}}\vert\le\nicefrac C{\rad^{1+\delta}}$ for some $\delta>0$ depending only on $\eta$ and $\ve$, we conclude that $\Vert\hspace{.05em}\eukzFundtrf\Vert_{\Lp^2(\M)} \le \nicefrac C{\Aradius^\delta}$, where $\eukzFundtrf*$ denotes the trace-free part of the Euclidean second fundamental form of $\outx(\M)\hookrightarrow(\R^3,\eukoutg*)$. In particular, we can use DeLellis-M\"uller's result \cite[Thm~1.1]{DeLellisMueller_OptimalRigidityEstimates} to conclude that $\M$ is a sphere and that there is a center point $\centerz\in\R^3$ and a function $f:\sphere^2_\Aradius(\centerz\,)\to\R$ such that
\begin{equation*} \M = \graph f, \qquad \Vert f\Vert_{\Hk^2(\sphere^2_\Aradius(\centerz\,))} \le C\,\Aradius^{2-\delta},\qquad\vert\centerz\,\vert\le\c_0\Aradius + C\,\Aradius^{1-\eta} \labeleq{Regularity_of_surfaces_in_asymptotically_flat_spaces__ineq_f__beta}. \end{equation*}
In particular, the metric $\g*$ on $\M$ is (approximately) equal to $\Aradius^2\sphg*$, where $\sphg*$ is the standard metric on the Euclidean unit sphere $\sphere^2$. As the Sobolev inequalities are satisfied on $\sphere^2$, we conclude that they are satisfied on $\M$ -- compare to \cite[Sec.~2]{christodoulou1993global}. Thus, \eqref{Regularity_of_surfaces_in_asymptotically_flat_spaces__ineq_f__beta} implies $\min\rad\ge\c_0\,\Aradius-C\,\Aradius^{1-\delta}$. We can therefore use Proposition \ref{Bootstrap_for_trace_free_second_fundamental_form} for $\kappa=\frac32+\outve$ to deduce \eqref{Regularity_of_surfaces_in_asymptotically_flat_spaces__ineq_k}. Using \cite[Thm~1.1]{DeLellisMueller_OptimalRigidityEstimates} once more, we get \eqref{Regularity_of_surfaces_in_asymptotically_flat_spaces__ineq_f}, where we use the pointwise estimates of the second fundamental form. By the same arguments as for the $\Lp^\infty$-estimates of $\zFundtrf*$ in Proposition \ref{Bootstrap_for_trace_free_second_fundamental_form}, we conclude $\zFund\in\Wkp^{1,p}(\M)$ implying $f\in\Ck^2(\sphere^2)$.
\end{proof}
Now, we can prove the central result that any $\mathcal C_\eta(\c_0)$-asymptotically concentric CMC-surface is stable. To do so, we first prove the central argument that the lowest eigenvalues of the Laplacian can be calculate using the Hawking mass. A comparable argument was (implicitly) used by Huang to conclude the second inequality in \eqref{Eigenvalues_of_the_Laplacian__d} \cite{Huang__Foliations_by_Stable_Spheres_with_Constant_Mean_Curvature}.
\begin{lemma}[Eigenvalues of the Laplacian]\label{Eigenvalues_of_the_Laplacian}
Let $(\outM,\outg*)$ be a three-dimen\-sio\-nal Riemannian manifold and let $\outve>0$, $\eta\in\interval0*1$ and $\c_0,\c_1\ge0$ be constants. Assume that $\M\in\mathcal A^{\ve,\eta}_\Aradius(\c_0,\c_1)$ is a closed hypersurface in $\outM$ with constant mean curvature $\H\equiv\mean\H(\M)=:\nicefrac{{-}2}\Hradius$ and Hawking mass $\mHaw:=\mHaw(\M)$. Then there are constants $\Hradius_0=\Cof{\Hradius_0}[\outve][\oc][\c_1][|\mHaw|]$ and $C=\Cof[\ve][\oc][\c_1][|\mHaw|]$ such that every complete orthonormal system $\{f_i\}_{i=0}^\infty$ of $\Lp^2(\M)$ by eigenfunctions $\eflap_i$ of the {\normalfont(\hspace{-.05em}\nolinebreak}negative{\normalfont)} Laplacian with corresponding eigenvalue $\ewlap_i$ with $\ewlap_i\le\ewlap_{i+1}$ satisfies
\begin{equation*} \vert\ewlap_i - (\frac2{\Hradius^2} - \frac{6\,\mHaw}{\Hradius^3} + \int_{\M}\outric*(\nu,\nu)\,f_i^2 \d\mug)\vert \le \frac C{\Hradius^{3+\ve}} \qquad\forall\,i\in\{1,2,3\} \labeleq{Eigenvalues_of_the_Laplacian__t} \end{equation*}
and
\begin{equation*} \ewlap_k > \frac 5{\Hradius^2} \quad\forall\,k > 3, \qquad \vert\int_{\M}\outric*(\nu,\nu)\,f_i\,f_j \d\mug \vert \le \frac C{\Hradius^{3+\ve}} \quad\forall\,i\neq j\in\{1,2,3\} \labeleq{Eigenvalues_of_the_Laplacian__d} \end{equation*}
if $\Hradius>\Hradius_0$.
\end{lemma}
\begin{proof}
Proposition \ref{Regularity_of_surfaces_in_asymptotically_flat_spaces} implies that there are coordinates $x:\M\to\sphere^2$ with
\[ \Vert x^*\sphg_{I\!J} - \g_{I\!J} \Vert_{\Hk^2(\M)} \le C\,\Hradius^{\frac12-\ve}, \]
where $\sphg*$ denotes the standard metric on the Euclidean sphere $\sphere^2_\Hradius(0)$ with radius $\Hradius$. This implies
\[ \Vert \eflap_i - \eflap[\euclideane]_i \Vert_{\Hk^2(\M)} \le \frac C{\Hradius^{\frac12+\ve}},\qquad \vert\ewlap_i - \ewlap[\euclideane]_i \vert \le \frac C{\Hradius^{\frac52+\ve}} \]
for any complete orthonormal system $\{\eflap_i\}_i$ of $\Lp^2(\M)$ of eigenfunctions of the (negative) Laplacian with corresponding eigenvalues $\ewlap_i\le\ewlap_{i+1}$, where $\eflap[\euclideane]_i$ denotes the push forward of the corresponding eigenfunctions of the (negative) Laplacian on the Euclidean sphere with corresponding eigenvalue $\ewlap[\euclideane]_i$. In particular, the first inequality in \eqref{Eigenvalues_of_the_Laplacian__d} holds and we know
\begin{gather*}
 \Vert \Hesstrf\eflap_i\Vert_{\Lp^2(\M)} \le \frac C{\Hradius^{\frac52+\outve}} \quad\forall\,i\in\{1,2,3\}, \qquad
 \vert\ewlap_i - \frac2{\Hradius^2} \vert \le \frac C{\Hradius^{\frac52+\outve}} \quad\forall\,i\in\{1,2,3\}, \labeleq{Eigenvalues_of_the_Laplacian_ewlap_inequ} \\
 \hspace{-.8em}\Vert\g*(\levi*\eflap_i,\levi*\eflap_j) - \frac{3\,\delta_{ij}}{\Hradius^2\,\volume{\M}}+\frac{\eflap_i\,\eflap_j}{\Hradius^2}\Vert_{\Lp^p(\M)} \le \frac C{\Hradius^{\frac92+\outve-\frac2p}} \quad\forall\,i,j\in\{1,2,3\},\;p\in\interval*1\infty.\hspace{-1em} \labeleq{Eigenvalues_of_the_Laplacian_levi_inequ}
\end{gather*}
Thus, after integration and integration by parts the Bochner-Lichnerowicz formula
\begin{align*}
 \frac{\laplace\g*(\levi*\eflap_i,\levi*\eflap_j)}2
	={}& \trtr{\Hess\eflap_i}{\Hess\eflap_j} + \frac12(\g*(\levi*\eflap_i,\levi*\laplace\eflap_j)+\g*(\levi*\laplace\eflap_i,\levi*\eflap_j)) \\
		& + \frac\sc2\g*(\levi*\eflap_i,\levi*\eflap_j) \\
	={}& \trtr{\Hesstrf\eflap_i}{\Hesstrf\eflap_j} + \frac{\ewlap_i\,\ewlap_j}2\,\eflap_i\,\eflap_j - \frac{\ewlap_i+\ewlap_j-\sc}2\,\g*(\levi*\eflap_i,\levi*\eflap_j)
\end{align*}
implies
\[ \vert\frac{\ewlap_i^2}2\,\delta_{ij} - \int\frac\sc2\,\g*(\levi*\eflap_i,\levi*\eflap_j)\d\mug \vert \le \frac C{\Hradius^{5+\ve}}\qquad\forall\,i,j\in\{1,2,3\}. \]
Using the Gau\ss-equation and the inequality $\vert\zFundtrf*\vert\le \nicefrac C{\Hradius^{\frac32+\outve}}$ from \eqref{Regularity_of_surfaces_in_asymptotically_flat_spaces__ineq_k}, we deduce
\[ \vert\ewlap_i^2\,\delta_{ij} - \int(\outsc-2\outric*(\nu,\nu))\g*(\levi*\eflap_i,\levi*\eflap_j)\d\mug - \frac2{\Hradius^2}\int \g*(\levi*\eflap_i,\levi*\eflap_j)\d\mug \vert \le \frac C{\Hradius^{5+\ve}}. \]
Integrating by parts and plugging in the above (asymptotic) characterization of $\levi*\eflap_i$, this implies
\begin{equation*} \vert\delta_{ij}(\ewlap_i^2-\frac2{\Hradius^2}\ewlap_i)\, - \frac1{\Hradius^2} \,\int(\outsc-2\outric*(\nu,\nu))\,(\frac{3\delta_{ij}}{\volume{\M}}-\eflap_i\,\eflap_j)\d\mug\vert \le \frac C{\Hradius^{5+\ve}}. \labeleq{Eigenvalues_of_the_Laplacian_eflap_inequation} \end{equation*}
In particular, the second inequality of \eqref{Eigenvalues_of_the_Laplacian__d} holds, too. We know
\[ \vert \mHaw - \frac\Hradius{16\pi}\int\outsc-2\outric*(\nu,\nu) \d\mug \vert \le \frac C{\Hradius^\ve} \]
due to the Gau\ss-Bonnet theorem, the Gau\ss\ equation, and the inequalities on $\zFundtrf*$. Thus, we get \eqref{Eigenvalues_of_the_Laplacian__t} by solving the inequality \eqref{Eigenvalues_of_the_Laplacian_eflap_inequation} if we keep \eqref{Eigenvalues_of_the_Laplacian_ewlap_inequ} ($\ewlap_i\approx\nicefrac2{\Hradius^2}$) in mind. Hence, all claims of this lemma are proven.
\end{proof}%
It is well-known that the eigenvalues of the stability operator $\jacobiext*$ of a CMC-surface $\M\in\mathcal A^{\ve,\eta}_\Aradius(\c_0,\c_1)$ are of order $\Hradius^{-2}$ except for three eigenvalues of order $\Hradius^{-3}$, where the stability operator of $\M$ is the linearization of the \emph{mean curvature map}. It is characterized by
\[ \jacobiext*f = \laplace f + (\outric*(\nu,\nu) + \trtr\zFund\zFund)\,f \qquad\forall\,f\in\Hk^2(\M), \]
for more details see Proposition \ref{Stability} and Section 3 in this work, \cite{huisken_yau_foliation,metzger2007foliations} or (in a more general context) \cite{barbosa2012stability}. As we will see in Proposition \ref{Stability}, the corresponding partition of $\Hk^2(\M)$ (respectively $\Lp^2(\M)$) is (asymptotically) given as follows.
\begin{definition}[Translational and deformational part of a function]
Assume that $\M\in\mathcal A^{\ve,\eta}_\Aradius(\c_0,\c_1)$ is a closed hypersurface in $\outM$ with constant mean curvature $\H\equiv\mean\H(\M)=:\nicefrac{{-}2}\Hradius$ and Hawking mass $\mHaw:=\mHaw(\M)$. The \term{translational part} $\trans f$ of a function $f\in\Lp^2(\M)$ is the $\Lp^2(\M)$-orthogonal projection on the linear span of eigenfunctions of the (negative) Laplacian with eigenvalue $\lambda$ satisfying $\vert\lambda-\nicefrac2{\Hradius^2}\vert\le\nicefrac3{\Hradius^2}$, i.\,e.
\[ \trans g := \sum_{\vert\ewlap_i+\nicefrac2{\Hradius^2}\vert\le\nicefrac3{\Hradius^2}} \eflap_i\,\int_{\M} g\,\eflap_i \d\mug \qquad\forall\,g\in\Lp^2(\M), \]
where $\eflap_i$ and $\ewlap_i$ are defined as in Lemma \ref{Eigenvalues_of_the_Laplacian}. The \term{deformational part} $\deform g$ of such a function $g\in\Lp^2(\M)$ is defined by $\deform g := g-\trans g$.
\end{definition}
In Proposition \ref{Movement_of_the_spheres_by_the_lapse_function}, we explain the reason for calling these terms \term{translational} and \term{deformational part}, respectively.
Now, we can prove the announced stability proposition which is one of the central tools for the proofs of the main theorems.
\begin{proposition}[Stability]\label{Stability}
Let $(\outM,\outg*)$ be a three-dimensional Riemannian manifold and let $\outve>0$, $\eta\in\interval0*1$, $\c_0\in\interval*01$, and $\c_1\ge0$ be a constant. Assume that $\M\in\mathcal A^{\ve,\eta}_\Aradius(\c_0,\c_1)$ is a closed hypersurface in $\outM$ with constant mean curvature $\H\equiv\mean\H(\M)=:\nicefrac{{-}2}\Hradius$ and Hawking mass $\mHaw:=\mHaw(\M)$. Then there are constants $\Hradius_0=\Cof{\Hradius_0}[\outve][\oc][\c_0][\c_1][\eta][|\mHaw|]$ and $C=\Cof[\outve][\oc][\c_0][\c_1][\eta][|\mHaw|]$ such that
\begin{align*}
 \vert \int_{\M} (\jacobiext*\trans g)\,\trans h\d\mug - \frac{6\mHaw}{\Hradius^3}\int_{\M}\trans g\,\trans h\d\mug \vert \le{}& \frac C{\Hradius^{3+\ve}}\Vert\trans g\Vert_{\Lp^2(\M)}\,\Vert\trans h\Vert_{\Lp^2(\M)}, \labeleq{Stability__t} \\
 \Vert \deform g\Vert_{\Lp^2(\M)} \le{}& \Hradius^2 \Vert \jacobiext*\deform g\Vert_{\Lp^2(\M)} \labeleq{Stability__d}
\end{align*}
for every $g,h\in\Hk^2(\M)$ if $\Hradius>\Hradius_0$. If $\efjac\in\Hk^2(\M)$ is a eigenfunction of ${-}\jacobiext*$ with corresponding eigenvalue $\ewjac$ then
\begin{equation*} \vert\,\ewjac\vert \ge \frac3{2\Hradius^2} \qquad\text{or}\qquad\Vert\,\deform{\efjac}\Vert_{\Hk^2(\M)} \le \frac C{\Hradius^{\frac12+\outve}}\Vert\,\efjac\Vert_{\Hk^2(\M)},\quad \vert\,\ewjac-\frac{6m}{\Hradius^3}\vert \le \frac C{\Hradius^{3+\outve}}. \labeleq{Stability__ewjac} \end{equation*}
Furthermore, the corresponding $\Wkp^{2,p}$-inequalities
\[ \Vert\trans g\Vert_{\Wkp^{2,p}(\M)} \le (\frac{\Hradius^3}{6\,\mHaw}+C\,\Hradius^{3-\outve})\,\Vert\jacobiext*g\Vert_{\Lp^p(\M)}, \quad
	\Vert\deform g\Vert_{\Wkp^{2,p}(\M)} \le C\,\Hradius^2\,\Vert\jacobiext*g\Vert_{\Lp^p(\M)} \]
hold for every function $g\in\Wkp^{2,p}(\M)$ and $p\in\interval*2\infty$ if $\Hradius>\Hradius_0$.
\end{proposition}
\begin{proof}
Using Proposition \ref{Regularity_of_surfaces_in_asymptotically_flat_spaces}, we conclude \eqref{Stability__t} and \eqref{Stability__d} by \eqref{Eigenvalues_of_the_Laplacian__t} and the first inequality in \eqref{Eigenvalues_of_the_Laplacian__d}, respectively. Now, let $\efjac$ be a eigenfunction of ${-}\jacobiext*$ with eigenvalue $\ewjac$ satisfying $\vert\,\ewjac\vert\le\nicefrac3{(2\Hradius^2)}$. Without loss of generality, we assume $\Vert\efjac\Vert_{\Lp^2(\M)}=1$. Using the characterization of $\jacobiext*$ and \eqref{Regularity_of_surfaces_in_asymptotically_flat_spaces__ineq_k}, we see
\begin{equation*}
 \frac3{2\sigma^2}\,\vert\fint\efjac\d\mug\vert \ge\vert\ewjac\vert\,\vert\fint\efjac\d\mug\vert \ge \frac2{\Hradius^2}\vert\fint\efjac\d\mug\vert - \frac C{\Hradius^{\frac52+\outve}}\labeleq{Stability_mean_part}
\end{equation*}
and \eqref{Stability__t} implies
\begin{equation*}
 \vert\,\ewjac\vert
	\ge \vert\int\laplace\deform{\efjac}\,\deform{\efjac} \d\mug\vert - \frac2{\Hradius^2} \Vert\,\deform\efjac\Vert_{\Lp^2(\M)}^2 - \frac C{\Hradius^{\frac52+\outve}}. \labeleq{Stability_deform_part}
\end{equation*}
In particular, we get $\Vert\,\deform{\efjac}\Vert_{\Lp^2(\M)}^2\le\nicefrac34$ due to the first inequality in \eqref{Eigenvalues_of_the_Laplacian__d}. On the other hand again using \eqref{Stability__t}, we see
\[ \vert\,\ewjac\vert\,\Vert\,\trans{\efjac}\Vert_{\Lp^2(\M)}^2 = \vert\int \jacobiext\efjac\,\trans{\efjac}\d\mug \vert \le \frac{7\vert\mHaw\vert}{\Hradius^3}\,\Vert\,\trans{\efjac}\Vert_{\Lp^2(\M)}^2 + \frac C{\Hradius^{\frac52+\outve}}\,\Vert\,\trans{\efjac}\Vert_{\Lp^2(\M)}\,\Vert\,\deform{\efjac}\Vert_{\Lp^2(\M)}. \]
Thus, we get
\[ \frac{\vert\,\ewjac\vert}4 \le \vert\,\ewjac\vert\,\Vert\,\trans{\efjac}\Vert_{\Lp^2(\M)}^2 \le \frac{7\vert\mHaw\vert}{\Hradius^3}\,\Vert\,\trans{\efjac}\Vert_{\Lp^2(\M)}^2 + \frac C{\Hradius^{\frac52+\outve}} \le \frac C{\Hradius^{\frac52+\outve}} \]
implying
\[ \vert\,\ewjac\vert \le \frac C{\Hradius^{\frac52+\outve}}. \]
We know $\vert\fint\efjac\d\mug\vert \le \nicefrac C{\Hradius^{\frac12+\outve}}$ due to \eqref{Stability_mean_part} and therefore get
\[ \vert\,\ewjac\vert - \frac{6\vert\mHaw\vert}{\Hradius^3}\Vert\,\trans\efjac\Vert_{\Lp^2(\M)}^2 \ge \frac3{\Hradius^2} \Vert\,\deform\efjac\Vert_{\Lp^2(\M)}^2 - \frac C{\Hradius^{\frac52+\outve}}\,\Vert\,\deform\efjac\Vert_{\Lp^2(\M)} \]
by using \eqref{Stability__t} and calculating as in \eqref{Stability_deform_part}. Thus, we have proven the first inequality of the second case in \eqref{Stability__ewjac} as the above implies
\[ \Vert\,\deform{\efjac}\Vert_{\Lp^2(\M)} \le \frac C{\Hradius^{\frac12+\outve}}. \]
Therefore, we get \eqref{Stability__ewjac} by
\begin{align*}
 \vert\,\ewjac - \frac{6\mHaw}{\Hradius^3}\vert\,\Vert\,{\trans\efjac}\Vert_{\Lp^2(\M)}^2
	={}& \vert\int(\jacobiext*\efjac-\frac{6\mHaw}{\Hradius^3}\efjac)\,\trans\efjac\d\mug \vert\\
	\le{}& \vert\int\outric*(\nu,\nu)\,\deform{\efjac}\,\trans{\efjac} \d\mug \vert+ \frac C{\Hradius^{3+\outve}} \le \frac C{\Hradius^{3+\outve}},
\end{align*}
where we again used \eqref{Stability__t}.\smallskip

As $(\M,\g*)$ is almost a round sphere due to Proposition \ref{Regularity_of_surfaces_in_asymptotically_flat_spaces}, we can use the $\Wkp^{2,p}(\sphere^2_\sigma(\centerz\,))$-regularity for the Euclidean Laplacian on the Euclidean sphere (the Calderon-Zygmund estimates) to get the claimed $\Wkp^{2,p}(\M)$ inequalities.
\end{proof}
\section{Existence of the CMC-foliation}\label{existence_of_the_CMC-foliation}
In this section, we prove the existence and uniqueness of the CMC-foliation as well as the uniqueness and stability of the leaves of the foliation. First, let us state the existence result.
\begin{theorem}[Existence of a CMC-foliation]\label{Existence_of_a_CMC-foliation}
Let $(\outM,\outg,\outx)$ be a three-di\-men\-sio\-nal $\Ck^2_{\frac12+\ve}$-asymptotically flat Riemannian manifold and assume that its mass is non-vanishing, i.\,e.\ $\mass\neq0$. There is a constant $\Hradius_0=\Cof{\Hradius_0}[|\mass|][\outve][\oc]$, a compact set $\outsymbol L\subseteq\outM$ and a diffeomorphism $\outPhi:\interval{\sigma_0}\infty\times\sphere^2\to\outM\setminus\outsymbol L$ such that each of the surfaces $\M<\Hradius>:=\outPhi(\Hradius,\sphere^2)$ has constant mean curvature $\H<\Hradius>\equiv\nicefrac{{-}2}\Hradius$.
\end{theorem}
Furthermore, our proof of this existence result includes that the inequalities
\[
 \Vert\zFundtrf<\Hradius>\Vert_{\Hk(\M<\Hradius>)} \le \frac C{\Hradius^{\frac32+\outve}}, \quad
 \vert\centerz<\Hradius>\,\vert \le C\,\Hradius^{1-\ve}, \quad
 \vert\ewjac<\Hradius>_j\vert \ge \frac 3{2\Hradius^2}, \qquad\forall\,j\ge4
\]
hold for some constant $C=\Cof[|\mass|][\outve][\oc]$ not depending on $\Hradius>\Hradius_0$, where $\zFundtrf$, $\ewjac<\Hradius>_j$, and $\centerz<\Hradius>$ denote the trace-free part of the second fundamental form, the $j$-the eigenvalue of $\jacobiext*$, and the Euclidean coordinate center
\[ \centerz<\Hradius>_i := \fint_{\M<\Hradius>} \outx_i \d\mug \]
of $\M<\Hradius>$ constructed in Theorem \ref{Existence_of_a_CMC-foliation}, respectively. In particular, combining these inequalities with Proposition \ref{Stability}, we conclude the following stability result for the leaves.
\begin{theorem}[Stability of the foliation]\label{Stability_of_the_foliation}
Let $(\outM,\outg,\outx)$ be a three-di\-men\-sio\-nal $\Ck^2_{\frac12+\ve}$-asymptotically flat Riemannian manifold and assume that its mass is non-vanishing, i.\,e.\ $\mass\neq0$. Furthermore, let $\M\hookrightarrow\outM$ be the leaf of the CMC-foliation constructed in Theorem \ref{Existence_of_a_CMC-foliation} with constant mean curvature $\H\equiv\nicefrac{{-}2}\Hradius$ with $\Hradius>\Hradius_0$. Then there are constants $\c_1=\c_1(|\mass|,\outve,\oc)$ and $C=C(|\mass|,\outve,\oc)$ not depending on $\Hradius$ such that $\M\in\mathcal A^{\outve,\outve}_\Aradius(0,\c_1)$ and
\[ \vert\ewjac<\Hradius>_i-\frac{6\,\mass}{\Hradius^3} \vert \le \frac C{\Hradius^{3+\ve}} \qquad\forall\,i\in\{1,2,3\}, \]
where $\ewjac<\Hradius>_i$ {\normalfont(}$i\in\{1,2,3\}${\hspace{.05em}\normalfont)} denote the three eigenvalues of the {\normalfont(}negative{\hspace{.05em}\normalfont)} stability operator of $\M$ with the smallest absolute value. This surface $\M$ is stable as a CMC-surface if and only if $\mass>0$.
\end{theorem}
Besides these existence and stability results, we get the corresponding uniqueness theorem for the leaves of the foliation within a specific class of CMC-surfaces.
\begin{theorem}[Uniqueness of the CMC-surfaces]\label{Uniqueness_of_the_CMC_surfaces}
Let $\c_0\in\interval*01$, $\c_1\ge0$ and $\eta>0$ be constants and $(\outM,\outg,\outx)$ be a three-dimensional $\Ck^2_{\frac12+\ve}$-asympto\-ti\-cally flat Riemannian manifold and assume that its mass is non-vanishing, i.\,e.\ $\mass\neq0$. There is a constant $\Hradius_0=\Cof{\Hradius_0}[|\mass|][\outve][\oc][\c_0][\c_1][\eta]$ such that all hypersurfaces $\M\in\mathcal A^{\ve,\eta}_\Aradius(\c_0,\c_1)$ of $\outM$ with constant mean curvature $\H\equiv\mean\H(\M)=:\nicefrac{{-}2}\Hradius$ and $\Hradius>\Hradius_0$ coincide. \pagebreak[3]
\end{theorem}%
The proof of the existence and uniqueness theorem have the same structure as Metzger's proof of the same statements in the setting of a $\Ck^2_{1+\outve}$-asymptotically Schwarz\-schildean manifold \cite{metzger2007foliations}. Hence, we briefly explain his proof:

He defines the metrics
\[ \outg[\atime]* := \schwarzoutg* + \atime\,(\outg*-\schwarzoutg*) \quad\forall\,\atime\in\interval*0*1 \qquad\text{with}\quad\schwarzoutg* := (1+\frac\mass{2\rad})^4\,\eukoutg*\]
and proves that the interval $\intervalI\subseteq\interval*0*1$ of all \term*{artificial times} $\atime$ for which the theorem is true (for $\outg[\atime]*$ instead of $\outg*$) is non-empty, open and closed in (and thus equal to) $\interval*0*1$. It is well-known that there exists a (unique) CMC-foliation with respect to the Schwarzschild metric $\schwarzoutg*$, which means $\intervalI\supseteq\{0\}$. Now, he proves that $\intervalI$ is closed by a simple convergence argument and that $\intervalI$ is open by using the implicit function theorem: For every surface $\M[\atime]<\Hradius>$ which has constant mean curvature $\H\equiv\nicefrac{{-}2}\Hradius$ with respect to $\outg[\atime]*$, he defines the mean curvature map
\[ \textsf H : \Wkp^{2,p}(\M)\to\Lp^p(\M) : f \mapsto \H(\graph f), \]
where $\H(\graph f)$ denotes the mean curvature of the graph of $f$ and $p>2$ is arbitrary. Recalling that the Fr\'echet derivative of this map at $f\equiv0$ is the stability operator, he concludes by the inverse function theorem that $\intervalI$ contains a neighborhood of $\atime$ if the stability operator is invertible. Thus, $\intervalI$ is open if the stability operator of every surface $\M<\Hradius>[\atime]$ is invertible. Furthermore, he needs to control the Euclidean coordinate center $\centerz[\atime]<\Hradius>$ in order to use the assumed decay rate of $\outg*-\schwarzoutg*$.\smallskip

We copy the explained proof structure and replace two main arguments: We conclude the invertibility of the stability operator of the surfaces from the arguments used in Proposition \ref{Stability} (instead of using the concrete form of the Ricci curvature in Schwarz\-schild) and control the Euclidean coordinate center by estimating its $\atime$-derivative. For the latter, we use a trick used multiple times in the literature to prove that the ADM- (or CMC-)center of mass is well-defined (under stronger assumptions than we assume here), see for example \cite{Huang__Foliations_by_Stable_Spheres_with_Constant_Mean_Curvature,metzger_eichmair_2012_unique}.

\begin{assumptions}[Existence and regularity intervals]
Let $\Hradius>0$, $\c_0\in\interval*01$, $\c_1>0$ and $\eta\in\interval0*1$ be constants. Assume that $\Phi:\intervalI\times\sphere^2\to\outM$ is a $\Ck^1$-map such that
\begin{enumerate}[nolistsep,label={{\normalfont(\arabic{*})}}]
\item $\intervalI\subseteq\interval*0*1$ is a interval with $0\in\intervalI$;
\item $\M[\atime]:=\Phi(\atime,\sphere^2)$ has constant mean curvature $\H[\atime]\equiv\nicefrac{{-}2}\Hradius$ with respect to $\outg[\atime]*:=\schwarzoutg*+\atime\,(\outg*-\schwarzoutg*)$;
\item $\M[0]=\sphere^2_{\rradius(\Hradius)}(0)$ for the specific radius $\rradius(\Hradius)$ for which $\H[0]<\rradius(\Hradius)>\equiv\nicefrac{{-}2}\Hradius$;
\item $\partial*_\atime\Phi$ is orthogonal to $\M[\atime]$ for any $\atime\in\intervalI$.
\end{enumerate}
Furthermore, let $\intervalI$ be maximal, i.\,e.\ if $\Phi':\intervalI'\times\sphere^2\to\outM$ satisfies the above assumptions for the same $\Hradius$, then $\intervalI'\subseteq\intervalI$.

Let $\intervalJ\subseteq\intervalI$ be the maximal subset such that $\M[\atime]\in\mathcal A^{\ve,\eta}_\Aradius(\c_0,\c_1)$ for every $\atime\in\intervalJ$.\pagebreak[3]
\end{assumptions}
If we choose $\Hradius$ and $\c_1$ sufficiently large and $\eta\in\interval0*1$ sufficiently small then such a $\Phi$ exists for some $\intervalJ\supseteq\{0\}$, because $\outg[0]*$ is the Schwarzschild metric with mass $\mass\neq0$. Now, we first show that $\intervalI$ contains a neighborhood of $\intervalJ$ in $\interval*0*1$, then that $\intervalJ$ is open and closed in $\intervalI$, i.\,e.\ $\intervalJ=\intervalI$. This implies that $\intervalI$ is open in $\interval*0*1$ and a simple convergence argument finishes the proof that $\intervalI=\interval*0*1$.
\begin{lemma}[\texorpdfstring{$\intervalI$}I is a neighborhood of \texorpdfstring{$\intervalJ$}J]\label{I_is_a_neighborhood_of_J}
There is a constant $\Hradius_0=\Cof{\Hradius_0}[|\mass|][\outve][\oc][\c_0][\c_1][\eta]$ such that $\intervalI$ contains a neighborhood of $\intervalJ$ in $\interval*0*1$ if $\Hradius>\Hradius_0$.
\end{lemma}
\begin{proof}
Let $\atime_0\in\intervalJ$ be arbitrary and suppress the corresponding index. Proposition \ref{Stability} implies that the stability operator $\jacobiext*:\Hk^2(\M)\to\Lp^2(\M)$ of $\M$ is invertible if $\Hradius$ is sufficiently large and the Hawking mass om $\M$ does not vanish. We prove the latter in Lemma \ref{Ric(nu,nu)-integrals} and assume that $\Hradius$ is sufficiently large. Thus, if we fix $p\in\interval2\infty$, then the $\Wkp^{2,p}(\M)$-regularity of the stability operator implies that the restriction $\jacobiext*:\Wkp^{2,p}(\M)\to\Lp^p(\M)$ is invertible. This (invertible) operator is the Fr\'echet derivative of the mean curvature map
\[ \textsf H : \interval*0*1\times\Wkp^{2,p}(\M) \to \Lp^p(\M) : (\atime,f)\mapsto \H[\atime](\graph f) \]
with respect to the second component at $(\atime,f)\equiv(\atime_0,0)$, where $\M[\atime](\graph f)$ denotes the mean curvature of the graph of $f$ with respect to the metric $\outg[\atime]*$. This map is well-defined for $f$ with sufficiently small $\Vert f\Vert_{\Wkp^{2,p}(\M)}$ due to Proposition \ref{Regularity_of_surfaces_in_asymptotically_flat_spaces}. Thus, the implicit function theorem implies that there is a $\eta>0$ and a $\Ck^1$-map $\gamma:\interval{\atime_0-\eta}{\atime_0+\eta}\to\Wkp^{2,p}(\M)$ such that $\textsf H(\atime,\gamma(\atime))=\textsf H(\atime,\gamma(0))\equiv\nicefrac{{-}2}\Hradius$ and that this map is uniquely defined by this property -- at least within a neighborhood of $0\in\Wkp^{2,p}(\M)$. In particular, we can extend $\Phi$ to a neighborhood of $\atime_0$. Hence, the assumed maximality of $\intervalI$ implies that $\intervalI$ contains a neighborhood of $\atime_0$.%
\end{proof}%
Analyzing the proof of this lemma, we see furthermore that $\Phi$ is uniquely defined by the assumed four properties -- at least in a neighborhood of $\intervalJ$. Furthermore, we see that $\outx\circ\Phi$ is differentiable as a map from $\intervalI$ to $\Wkp^{2,p}(\sphere^2;\R^3)$.

In order to prove that $\intervalI$ is open, we have to show that all surfaces $\M[\atime]$ satisfy the assumptions of \ref{Stability}, i.\,e.\ that $\intervalI=\intervalJ$. To do so, we use again a open-closed argument. We find that $\intervalJ$ is closed as all assumptions on $\atime\in\intervalJ$ are closed assumptions (non-strict inequalities) and the corresponding quantities depend continuously on $\Phi$ due to the differentiability of $\Phi$ explained in the last paragraph. Thus, we only have to prove that $\intervalJ$ is open within $\intervalI$.
\begin{lemma}[\texorpdfstring{$\intervalJ=\intervalI$}{J=I}]\label{I=J}
There are constants $\Hradius_0=\Cof{\Hradius_0}[|\mass|][\outve][\oc]$ and ${\c_1}'=\Cof{{\c_1}'}[|\mass|][\outve][\oc]$ such that $\intervalJ=\intervalI$ if $\Hradius>\Hradius_0$, $\c_1\ge{\c_1}'$, and $\eta\le\ve$, i.\,e.\ $\M[\atime]\in\mathcal A^{\outve,\outve}_\Aradius(0,{\c_1}')$ for every $\atime\in\intervalI$ if $\sigma>\sigma_0$.
\end{lemma}
\begin{proof}
Fix $\atime_0\in\intervalJ$ and suppress the corresponding index. As explained aboveby the continuity of $\Phi$, we can assume that there is a neighborhood of $\atime_0$ in $\intervalI$ such that $\M[\atime]$ satisfies the assumptions of Proposition \ref{Regularity_of_surfaces_in_asymptotically_flat_spaces} for altered constants $\c_0'=\nicefrac{(1+\c_0)}2$, $\c_1'=2\c_1$, and $\eta'=\nicefrac\eta2$. Thus, it is sufficient to prove that any surfaces satisfying these assumptions for \emph{some} constants $\c_1'$, $\c_2'$ and $\eta'$ satisfy these for specific constants $\c_1$, $\c_2$ and $\eta$ only depending on $\outve$ and $\oc$. To prove this, we show estimates for the derivatives of the quantities controlled by these constants.

Denote by $\rnu$ the lapse function of $\Phi$, i\,e.
\[ \rnu := \outg[\atime_0]*(\partial[\atime]<\atime=\atime_0>@\Phi,\nu), \]
where $\nu$ is the outer unit normal of $\M\hookrightarrow(\outM,\outg[\atime_0]*)$. We see that
\[ 0 \equiv \partial[\atime]<\atime=\atime_0>@{(\H[\atime])} = \partial[\atime]<\atime=\atime_0>@{(\H[\atime_0](\M[\atime]))} + \partial[\atime]<\atime=\atime_0>@{(\H[\atime](\M[\atime_0]))} = \jacobiext*\rnu + \partial[\atime]<\atime=\atime_0>@{(\H[\atime](\M[\atime_0]))}, \]
where $\H[\atime](\M')$ denotes the mean curvature of any hypersurface $\M'\hookrightarrow\outM$ with respect to the metric $\outg[\atime]*$. This means for $2\,\outzFund*:=\schwarzoutg*-\outg*$ and $\outmomden*:=\outdiv(\outtr\,\outzFund-\outzFund*)$ that
\[ \jacobiext*\rnu = {-}\partial[\atime]<\atime=\atime_0>@{\H[\atime](\M[\atime_0])} = {-}\outmomden*(\nu) + \div\outzFund_\nu - \trtr\zFund\outzFund, \]
see \cite[Prop.~3.7]{nerz2013timeevolutionofCMC} for $\uniM:=\interval*0*1\times\outM$ and $\unig:={-}\d\atime^2+\outg[\atime]*$.\footnote{The artificial quantities $\outzFund*$ and $\outmomden*$ are actually the second fundamental form and the momentum density of $\{\atime_0\}\times\outM\hookrightarrow(\uniM,\unig)$.}\pagebreak[1] In particular, the assumptions on $\outg*$ imply that
\[ \Vert\jacobiext*\rnu\Vert_{\Lp^\infty(\M)} \le \frac C{\Hradius^{\frac32+\outve}}. \]
Thus, Proposition \ref{Stability} and the regularity of the Laplacian ensure that there is a constant $C$ such that
\[ \Vert\deform\rnu\Vert_{\Wkp^{2,p}(\M)} + \frac1\Hradius\Vert\trans\rnu\Vert_{\Wkp^{2,p}(\M)} \le C\,\Hradius^{\frac12-\outve+\frac p2} \qquad\forall\,p<\infty. \]
In particular, we conclude that
\[ \vert\partial[\atime]@{\volume{\M[\atime]}}\vert = \vert\H\int_{\M}\rnu\d\mug\vert \le C\,\Hradius^{\frac32-\outve}, \qquad \vert\partial[\atime](\int_{M<\atime>}\H^2\d\mug)\vert \le \frac C{\Hradius^{\frac12+\outve}}, \]
because $\int\trans\rnu\d\mug=0$. In particular, the derivative of $\int\H^2\d\mug$ is controlled sufficiently and we only have to prove $\vert\partial*_\atime{(\outx_i\circ\Phi)}\vert\le C\,\Hradius^{1-\outve}$. Hence, it is sufficient to verify
\begin{equation*} \Vert\trans\rnu\Vert_{\Hk^2(\M)} \le C\,\Hradius^{2-\outve}. \labeleq{I=J_transrnu}\end{equation*}
Let $\{\eflap_i\}_{i=0}^\infty$ again denote a complete $\Lp^2$-orthonormal system of eigenfunctions of the (negative) Laplacian with corresponding eigenvalues $\ewlap_i$. Per definition of $\trans\rnu$, we get \eqref{I=J_transrnu} if we validate
\[ \vert\int\rnu\,\eflap_i \d\mug\vert \le C\,\Hradius^{2-\ve} \qquad\forall\,i\in\{1,2,3\}. \]
Using Proposition \ref{Stability}, this is equivalent to prove
\[ \frac C{\Hradius^{1+\ve}} \ge \vert\int\rnu\;\jacobiext\eflap_i \d\mug\vert = \vert\int(\outmomden*(\nu) - \div\outzFund_\nu + \trtr\zFund\outzFund)\,\eflap_i\d\mug\vert. \]
Now, we use Proposition \ref{Regularity_of_surfaces_in_asymptotically_flat_spaces} and see that
\[ \Vert\deform{\nu_i}\Vert_{\Hk^2(\M)} \le \frac C{\Hradius^{\frac12+\ve}}\Vert\nu_i\Vert_{\Lp^2(\M)} \]
by comparing with the Euclidean sphere. Thus, we can replace $\eflap_i$ by $\nu_i$ and only have to show that
\[ \frac C{\Hradius^{1+\ve}} \ge \vert\int(\outmomden*(\nu) - \div\outzFund_\nu + \trtr\zFund\outzFund)\,\nu_i\d\mug\vert. \]
This is a technical calculation done in Lemma \ref{Estimating_the_center__implicit}. Hence, we conclude that
\[ \vert\partial[\atime](\int\H^2\d\mug)\vert \le \frac C{\Hradius^{\frac12+\outve}}, \qquad\vert\partial[\atime]@{(\outx_i\circ\Phi)}\vert\le C\,\Hradius^{1-\outve}. \]
By the arguments explained at the beginning of this proof, we get that $\intervalJ$ is open in $\intervalI$. As explained before this lemma, this implies that $\intervalI=\intervalJ$.
\end{proof}
Now, we can finally prove that $\intervalI=\interval*0*1$. In particular, there exists a surface $\M<\Hradius>$ with constant mean curvature $\H<\Hradius>\equiv\nicefrac{{-}2}\Hradius$ with respect to $\outg*$.
\begin{lemma}[\texorpdfstring{$\intervalI=\interval*0*1$}{I=[0,1]}]\label{I=[0;1]}
There is a constant $\Hradius_0=\Cof{\Hradius_0}[|\mass|][\outve][\oc]$ such that $\intervalI=\interval*0*1$ if $\Hradius>\Hradius_0$.
\end{lemma}
\begin{proof}
Define $\atime_0:=\sup\intervalI$. By the argument of proof of Lemma \ref{I=J}, we know that for any $\atime\in\intervalI$ and $p\in\interval1\infty$
\[ \Vert\partial[\atime]@{(\outx\circ\Phi)}\Vert_{\Wkp^{1,p}(\sphere^2_\Hradius;\R^3)} = \Vert\rnu[\atime]\;\outx_*(\nu[\atime])\Vert_{\Wkp^{1,p}(\sphere^2_\Hradius;\R^3)} \le C\,\Hradius^{\frac32}. \]
This means that $\Phi:\intervalI\to\Wkp^{1,p}(\sphere^2_\Hradius;\R^3)$ is Lipschitz continuous. As every Lipschitz continuous function on an interval can continuously extended to the closure of the interval, this means that $\Phi$ can be extended to\footnote{or in case of $\atime_0\in\intervalI$ is} a Lipschitz continouse map $\Phi':\interval*0*{\atime_0}\to\Wkp^{1,p}(\sphere^2_\Hradius;\R^3)$. In particular, $\outg*$ induces a well-defined metric $\g[\atime_0\,]*$ on $\M[\atime_0]:=\Phi'(\atime_0,\sphere^2_\Hradius)$, i.\,e.\ $\Phi'(\atime_0)^*\,\g[\atime_0\,]*\in\Lp^2(\sphere^2)$ is a well-defined metric on $\sphere^2$. With
\[ \Vert\partial[\atime]@{(\Phi(\atime)^*\,\g[\atime]*)}\Vert_{\Hk(\sphere^2_\Hradius)} = \Vert\rnu[\atime]\;\zFund*\Vert_{\Hk(\sphere^2_\Hradius)} \le C\,\Hradius^{1-\ve}, \]
we see that $\Phi'(\atime_0)^*\,\g[\atime_0\,]*\in\Hk(\sphere^2)$ implying that the second fundamental form of $\M[\atime_0]$ is well-defined and $\M[\atime_0]$ has constant mean curvature $\H[\atime_0]\equiv\nicefrac{{-}2}\Hradius$ with respect to $\outg[\atime_0]*$. All in all, we get $\M[\atime_0]\in\mathcal A^{\outve,\outve}_\Aradius(\c_0',\c_1')$ for the constants $\c_0'=\Cof{\c_0'}[|\mass|][\outve][\oc]$ and $\c_1'=\Cof{\c_1'}[|\mass|][\outve][\oc]$ from Lemma \ref{I=J}. Using the same arguments as in Lemma \ref{I_is_a_neighborhood_of_J}, we conclude that $\Phi'$ is differentiable in $\atime_0$. By the maximality of $\intervalI$, this proves that $\atime_0\in\intervalI$.

All in all, the above arguments prove that $\intervalI$ is closed and the Lemmata \ref{I_is_a_neighborhood_of_J} and \ref{I=J} imply that $\intervalI$ is open in $\interval*0*1$. Thus, $\intervalI=\interval*0*1$.
\end{proof}
As we will use the uniqueness of the CMC-leaves in order to prove that they foliate $\outM$, let us first prove the uniqueness of these surfaces.
\begin{proof}[Proof of Theorem \ref{Uniqueness_of_the_CMC_surfaces}]
Assume that $\Phi:\intervalI\times\sphere^2\to\outM$ is a $\Ck^1$-map such that
\begin{enumerate}
\item $\intervalI\subseteq\interval*0*1$ with $0\in\intervalI$;
\item $\M[\atime]:=\Phi(\atime,\sphere^2)$ has constant mean curvature $\H[\atime]\equiv\nicefrac{{-}2}\Hradius$ with respect to $\outg[\atime]*:=\schwarzoutg*+\atime\,(\outg*-\schwarzoutg*)$;
\item $\M[1]=\M$,
\item $\partial*_\atime{\Phi}$ is orthogonal to $\M[\atime]$ for any $\atime\in\intervalI$,
\end{enumerate}
Furthermore, let $\Phi$ be maximal, i.\,e.\ if $\Phi':\intervalI'\times\sphere^2\to\outM$ satisfies the above assumptions for the same $\Hradius$, then $\intervalI'\subseteq\intervalI$. With the same arguments as above, we conclude that $\intervalI=\interval*0*1$ and $\M[\atime]\in\mathcal A^{\outve,\outve}_\Aradius({\c_0}',{\c_1}')$ for some constants $\c_0\in\interval*01$ and $\c_1\ge0$ and every $\atime\in\intervalI'$. Thus, $\M[0]$ has constant mean curvature $\nicefrac{{-}2}\Hradius$ with respect to the Schwarzschild metric $\outg[0]*$. This implies $\M[0]=\sphere^2_{\rradius}(0)$. We explained below Lemma \ref{I_is_a_neighborhood_of_J} that $\Phi$ is uniquely defined by $\Phi(0,\cdot)$. Thus $\Phi(1,\sphere^2)=\M$ is uniquely defined by $\H$.
\end{proof}
\begin{proof}[Proof of Theorem \ref{Existence_of_a_CMC-foliation}]
By Lemma \ref{I=[0;1]}, there is a constant $\Hradius_0$ and a map $\Phi : \interval*0*1\times\interval{\Hradius_0}\infty\times\sphere^2\to\outM$ such that $\M[\atime]<\Hradius>:=\Phi(\atime,\Hradius,\sphere^2)$ has constant mean curvature $\H[\atime]<\Hradius>\equiv\nicefrac{{-}2}\Hradius$ with respect to $\outg[\atime]*$. Furthermore, there is a constant $\c_1=\Cof{\c_1}[|\mass|][\outve][\oc]$ such that $\M[\atime]<\sigma>\in\mathcal A^{\outve,\outve}_\Aradius(0,{\c_1})$ for every $\atime\in\interval*0*1$ and $\sigma>\sigma_0$ due to Lemma \ref{I=J}. In particular, the stability operator is invertible and an argument as in Lemma \ref{I_is_a_neighborhood_of_J} ensures that we can choose $\Phi$ to be continuously differentiable, when we keep the uniqueness (Theorem \ref{Uniqueness_of_the_CMC_surfaces}) in mind.

The only thing left to prove is the foliation property of $\Phi[\atime]:=\Phi(\atime,\cdot,\cdot)$. Let $\rnu:=\outg*(\spartial[\sigma]@\Phi,\nu)$ denote the lapse function in $\Hradius$-direction. In particular, the foliation property holds if $\Vert \rnu-1\Vert_{\Hk^2(\M)}\le C\,\Hradius^{1-\outve}$. As in the proof of Lemma \ref{I=J}, we know that
\[ \jacobiext*(\rnu-1) = \partial[\sigma]@{\H[\atime]<\Hradius>} - \outric*(\nu,\nu)-\trtr\zFund\zFund = {-}\outric*(\nu,\nu) - \trtr\zFundtrf\zFundtrf. \]
By Proposition \ref{Regularity_of_surfaces_in_asymptotically_flat_spaces}, this implies $\vert\jacobiext*(\rnu-1) + \outric*(\nu,\nu) \vert \le \nicefrac C{\Hradius^{3+\ve}}$.
Again with the same arguments as in the proof of Lemma \ref{I=J}, it is sufficient to prove
\[ \vert\int\outric*(\nu,\nu)\,\nu_i \d\mug \vert \le \frac C{\Hradius^{1+\ve}} \]
to get the foliation property. This is a technical calculation done in Lemma \ref{Ric(nu,nu)nu_i-integrals}.
\end{proof}

\section{The center of mass}\label{the_centers_of_mass}
Huisken-Yau defined the \term{CMC-center of mass} by using the CMC-foliation. As explained in the introduction, there are other definitions of center of mass, as the one defined by Regge-Teitel\-boim \cite{regge1974role} and Beig-{\'O}\,Murchadha \cite{beig1987poincare} -- as this is defined as ADM-type of expression, we call it ADM-center of mass.
\begin{definition}[ADM- and CMC-center of mass]
For any asymptotically flat Riemannian manifold $(\outM,\outg*,\outx)$ the \term{ADM-center of mass} is defined by
\[ \outcenterz_{\text{ADM}}^i := \frac1{16\pi m}\lim_{\rradius\to\infty}\int_{\sphere^2_\rradius(0)} \sum_{j=1}^3 (\outx^i(\partial[\outx]_\oj@{\outg_{\oj\ok}}-\partial[\outx]_\ok@{\outg_{\oj\oj}})\outx^\ok-(\outg_\oj^\oi\hspace{.05em}\frac{\outx^\oj}\rradius-\outg_{\oj\oj}\hspace{.05em}\frac{\outx^\oi}\rradius)) \d\eukmug \]
if this limit exists.

If $(\outM,\outg*,\outx)$ additionally posses a CMC-foliation $\{\M<\Hradius>\}_{\Hradius>\Hradius_0}$ such that the mean curvature of the leaf $\M<\Hradius>$ is $\nicefrac{{-}2}\Hradius$, then the \term{CMC-center of mass} is defined by
\[ \outcenterz_{\text{CMC}} := \lim_{\Hradius\to\infty}\centerz<\Hradius>, \qquad \centerz<\Hradius>^i := \fint_{\M<\Hradius>} \outx^i \d\mug \]
if this limit exists, where $\M<\Hradius>$ is the CMC-leaf with mean curvature $\nicefrac{{-}2}\Hradius$ of the unique CMC-foliation. Cederbaum and the other constructed examples proving that this center is not well-defined for every $\mathcal C^2_{\frac12+\outve}$-asymptotically flat manifold \cite{cederbaumnerz2013_examples}.
\end{definition}
It was proven that under different assumptions these centers coincide, see among others \cite{Corvino__On_the_center_of_mass_of_isolated_systems,huang2009center,Huang__Foliations_by_Stable_Spheres_with_Constant_Mean_Curvature,metzger_eichmair_2012_unique,nerz2013timeevolutionofCMC}. In this section, we prove the same result under our assumptions, but (like the cited results) we need an asymptotic symmetry condition on metric and scalar curvature: the Regge-Teitel\-boim conditions \cite{regge1974role}.
\begin{definition}[Regge-Teitel\-boim conditions]\label{Regge-Teitelboim_conditions}
Let $(\outM,\outg*,\outx)$ be a $\Ck^2_{\frac12+\ve}$-as\-ymp\-to\-ti\-cally flat Riemannian manifold. It satisfies the \term{$\Ck^2_{1+\outve}$-Regge-Teitel\-boim conditions} if
\begin{align*}
 \,\!&\,\!\vert\outg_{ij}(\outx)-\outg_{ij}({-}\outx)\vert + \rad\,\vert\outlevi_{ij}^k(\outx)+\outlevi_{ij}^k({-}\outx)\vert \\ &\,+ \rad^2\,\vert\outric_{ij}(\outx)-\outric_{ij}({-}\outx)\vert + \rad^{\frac52}\,\vert\outsc(\outx)-\outsc({-}\outx)\vert \le\frac\oc{\rad^{1+\outve}}, \labeleq{RTC_g}
\end{align*}
for any $i,j,k\in\{1,2,3\}$. It satisfies the \term{$\Ck^2_{\frac32+\outve}$-Regge-Teitel\-boim conditions} if \eqref{RTC_g} is satisfied for $\nicefrac\oc{\rad^{\frac32+\outve}}$ instead of $\nicefrac\c{\rad^{1+\outve}}$.
\end{definition}
In the cited (and most of the) literature, the (original) $\Ck^2_{\frac32+\outve}$-Regge-Teitel\-boim conditions are used. However, we allow the more general $\Ck^2_{1+\outve}$ conditions to get a more general result.\pagebreak[3]

Let us state the main theorem of this section.
\begin{theorem}[The centers of mass]\label{The_centers_of_mass}
Let $(\outM,\outg,\outx)$ be a three-dimensional $\Ck^2_{\frac12+\ve}$-asymptotically\vspace{-.2em} flat Riemannian manifold and assume that its mass is non-vanishing, i.\,e.\ $\mass\neq0$. If $(\outM,\outg*,\outx)$ satisfies the $\Ck^2_{1+\outve}$\!-Regge-Teitel\-boim conditions, then the CMC-center of mass is well-defined if and only if the ADM-center of mass is well-defined and in this case the two centers coincide.
In particular, the CMC-center of mass is well-defined and coincides with the ADM-center of mass if $(\outM,\outg*,\outx)$ satisfies the $\Ck^2_{\frac32+\outve}$-Regge-Teitel\-boim conditions.
\end{theorem}%
In the proof of this theorem, we will show that the CMC-surfaces are asymptotically symmetric if the $\Ck^2_{1+\outve}$-Regge-Teitel\-boim conditions are satisfied and the center is well-defined. This was already proven by Huang, but she additionally assumed that $(\outM,\outg*,\outx)$ is $\Ck^5_{\frac12+\ve}$-asymptotically flat and that the $\Ck^2_{\frac32+\outve}$-Regge-Teitel\-boim conditions are satisfied \cite{huang2009center}. Let us state this result.
\begin{proposition}[Symmetry of the CMC-leaves]\label{Symmetry_of_CMC_surfaces}
Let $(\outM,\outg,\outx)$ be a three-di\-men\-sional $\Ck^2_{\frac12+\ve}$-asymptotically flat Riemannian manifold and assume that its mass is non-vanishing, i.\,e.\ $\mass\neq0$. If $(\outM,\outg*,\outx)$ satisfies the $\Ck^2_{1+\outve}$-Regge-Teitel\-boim conditions and the CMC-center of mass is well-defined, then the CMC-leaves are asymptotically symmetric, i.\,e.\ for each $p\in\interval1\infty$ there is a constant $C=\Cof[|\mass|][\ve][\oc][p]$ and for every $\Hradius>\Hradius_0$ there is a function $\graphf<\Hradius>\in\Ck^1(\sphere^2_\Hradius(0))$ such that
\[ \graph\graphf<\Hradius> = \M<\Hradius>, \quad
		\Vert\graphf<\Hradius>\Vert_{\Wkp^{2,p}(\sphere^2_\Hradius(0))} \le C\,\Hradius^{\frac12+\frac2p-\outve}, \quad
		\Vert\graphf<\Hradius>-\graphf<\Hradius>\circ\varphi\Vert_{\Hk^2(\sphere^2_\Hradius(0))} \le C\,\Hradius^{1-\outve}, \]
where $\M<\Hradius>$ is as in Theorem \ref{Existence_of_a_CMC-foliation} and where $\varphi:\R^3\to\R^3:x\mapsto{-}x$.\pagebreak[3]
\end{proposition}%
Note that we can replace $C\,\Hradius^{\frac12+\frac2p-\outve}$ and $C\,\Hradius^{1-\outve}$ in Proposition \ref{Symmetry_of_CMC_surfaces} by $C\,\Hradius^{\frac2p-\outve}$ and $C\,\Hradius^{\frac12-\outve}$ if we assume the $\Ck^2_{\frac32+\outve}$-Regge-Teitel\-boim conditions. This can be seen by simply repeating the proof using Proposition \ref{Symmetry_of_CMC_surfaces} instead of Proposition \ref{Regularity_of_surfaces_in_asymptotically_flat_spaces}.\pagebreak[2]\smallskip

In order to prove Proposition \ref{Symmetry_of_CMC_surfaces} and Theorem \ref{The_centers_of_mass}, we need the following simple characterization of the translation of the Euclidean coordinate center of a surface under a deformation. For example this was done in \cite[Prop.~3.11]{nerz2013timeevolutionofCMC} in the case of asymptotically Schwarzschildean Riemannian manifolds.
\begin{proposition}[Movement of the spheres by the lapse function]\label{Movement_of_the_spheres_by_the_lapse_function}
Let $(\outM,\outg,\outx)$ be a three-dimensional $\Ck^2_{\frac12+\ve}$-asymptotically flat Riemannian manifold, $\c_0\in\interval*01$, $\c_1\ge0$ and $\eta>0$ be constants, and $(\M,\g*)\hookrightarrow(\outM,\outg*)$ be a closed hypersurface with $\M\in\mathcal A^{\ve,\eta}_\Aradius(\c_0,\c_1)$. Furthermore, let $\psi:\interval-{\eta_0}{\eta_0}\times\M\to\outM$ be a deformation of $\M$, i.\,e.\ $\psi(0,\M)=\M$ and $\psi(\eta,\cdot):\M\to\outM$ is a diffeomorphism onto its image. Then there are constants $C=\Cof[|\mass|][{\outve}][\oc][\eta][\c_0][\c_1]$ and $\Hradius_0=\Cof{\Hradius_0}[|\mass|][{\outve}][\oc][\eta][\c_0][\c_1]$ neither depending on $\psi$ nor on $\Aradius$, such that the Euclidean coordinate centers $\centerz<\eta>:=(\fint_{\M<\eta>}\outx^i\d\mug)_{i=1}^3$ of the hypersurfaces $\M<\eta>:=\psi(\eta,\M)$ satisfy
\[ \vert \partial[\eta]<\eta=0>@{\centerz<\eta>^i} - 3\fint_{\M} \nu_i\lapse[\psi] \d\mug \vert \le \frac C{\Hradius^{\frac32+\outve}}\Vert\lapse[\psi]\Vert_{\Lp^2(\M)} \]%
if $\Hradius>\Hradius_0$, where $\lapse[\psi]:=\outg*(\partial*_\eta\psi,\nu)$ is the lapse function of $\psi$ and $\nu$ denotes the outer unit normal of $\M\hookrightarrow(\outM,\outg*)$. In particular, this translation is {\normalfont(}in highest order{\normalfont)} characterized by the translational part $\relax^{\psi\!}\trans{\rnu}$ of $\rnu[\psi]$.
\end{proposition}%
\begin{proof}
As the Euclidean coordinate center of any surface is invariant under diffeomorphisms of the surface, we can assume that $\partial*_\eta\psi=\lapse[\psi]\,\nu$ for some function $\lapse[\psi]$. For the desired inequality, we first approximate the derivative of the numerator ($\centerz<\eta>\,\volume{\M<\eta>}=\int_{\M<\eta>}\outsymbol x_i\d\mug$):
\begin{align*}\hspace{4.5em}&\hspace{-4.5em}
 \vert\partial[\eta]@{(\centerz<\eta>_i\,\volume{\M<\eta>})} - 3\int_{\M} \normal_i\,\lapse[\psi] \d\mug + \int_{\M} \H\,\lapse[\psi]\, z_i \d\mug\vert \\
  ={}& \vert\int_{\M} \partial[\eta]@{(\outsymbol x_i\circ\psi)} - \H\,\lapse[\psi]\,\outsymbol x_i - 3\,\normal_i\,\lapse[\psi] + \H\,\lapse[\psi]\, z_i \d\mug\vert \\
  ={}& \vert\int_{\M} \H\,\lapse[\psi]\,(\outsymbol x_i-\centerz_i) + 2\,\normal_i\,\lapse[\psi] \d\mug \vert
  \le C\,\Hradius^{\frac12-\outve}\,\Vert\lapse[\psi]\Vert_{\Lp^2(\M)}
\end{align*}
Using the Leibniz formula, we conclude the claimed inequality by
\begin{align*}
 \vert\partial[\eta]@{(\centerz<\eta>_i)} - 3\fint_{\M} \normal_i\,\lapse[\psi] \d\mug\d\mug\vert
  \le{}& \volume{\M}^{-1}\vert\partial[\eta]@{(\centerz<\eta>_i\,\volume{\M[\eta]})}
			+ \centerz_i\int\H\,\lapse[\psi]\d\mug
			- 3\int_{\M} \normal_i\lapse[\psi] \d\mug\vert\\
	\le{}& \frac C{\Hradius^{\frac32+\outve}}\,\Vert\lapse[\psi]\Vert_{\Lp^2(\M)}.\qedhere\pagebreak[3]
\end{align*}
\end{proof}
\begin{proof}[Proof of Theorem \ref{The_centers_of_mass} and Proposition \ref{Symmetry_of_CMC_surfaces}]
It is well-known that the ADM-center of mass is well-defined for every $\Ck^2_{\frac12+\ve}$-asymptotically flat manifold satisfying the $\Ck^2_{\frac32+\outve}$-Regge-Teitel\-boim conditions \cite{beig1987poincare,chrusciel2004mapping} (by the arguments below, this is also implied by Lemma \ref{Estimating_the_center__implicit__strong_version}). Thus, it is sufficient to prove the first part of the claim.\pagebreak[1]%

First, we prove Proposition \ref{Symmetry_of_CMC_surfaces}. To do so, let $\M[\atime]<\Hradius>$ denote the leaf with mean curvature $\H[\atime]<\Hradius>\equiv\nicefrac{{-}2}\Hradius$ with respect to the artificial metric $\outg[\atime]*=\schwarzoutg+\atime\,(\outg*-\schwarzoutg*)$ and let $\graphf[\atime]<\Hradius>\in\Hk^3(\sphere^2_\Hradius(\centerz[\atime]<\Hradius>\,))$ denote its graph function (see Proposition \ref{Regularity_of_surfaces_in_asymptotically_flat_spaces}).  Here, (asymptotic) symmetry means that
\begin{equation*} \vert\centerz[\atime]<\Hradius>\,\vert \le C, \qquad\Vert \graphf[\atime]<\Hradius> - \graphf[\atime]<\Hradius>\circ\psi \Vert_{\Hk^2(\sphere^2_\Hradius(\centerz[\atime]<\Hradius>\,))} \le C\,\Hradius^{1-\outve}, \labeleq{The_centers_of_mass__symmetry} \end{equation*}
where $\psi:\sphere^2_\Hradius(\centerz[\atime]<\Hradius>\,)\to\sphere^2_\Hradius(\centerz[\atime]<\Hradius>\,):x+\centerz[\atime]<\Hradius>\mapsto{-}x+\centerz[\atime]<\Hradius>$. As we know that this is true at (artificial) time $\atime=0$ ($\graphf[0]<\Hradius>\equiv0$), it is sufficient to prove that (asymptotic) symmetry is preserved under the (orthogonal) deformation $\Phi$ (see Section \ref{existence_of_the_CMC-foliation}), i.\,e.
\begin{equation*}
	\Vert \deform{\rnu[\atime]<\Hradius>}\circ\graphF - \deform{\rnu[\atime]<\Hradius>}\circ\graphF\circ\psi \Vert_{\Hk^2(\sphere^2_\Hradius(\centerz[\atime]<\Hradius>\,))} \le C\,\Hradius^{1-\outve}, \qquad
	\Vert \trans{\rnu[\atime]<\Hradius>}\circ\graphF \Vert_{\Hk^2(\sphere^2_\Hradius(\centerz[\atime]<\Hradius>\,))} \le C\,\Hradius, \labeleq{The_centers_of_mass__symmetry_preserved}
\end{equation*}
where $\graphF:\sphere^2_\Hradius(\centerz)\to\outM$ is the graph map corresponding to $f$ and $\rnu:=\outg[\atime]*(\partial*_\atime\Phi,\nu[\atime]<\Hradius>)$ is the lapse function in $\atime$-direction. Using Proposition \ref{Movement_of_the_spheres_by_the_lapse_function} and the estimates on $\rnu$ proved in Lemma \ref{I=J}, we see that the first inequality of \eqref{The_centers_of_mass__symmetry} is an implication of \eqref{The_centers_of_mass__symmetry_preserved}, too. Furthermore, we see that the combination of \eqref{The_centers_of_mass__symmetry_preserved} and Theorem \ref{The_centers_of_mass} also proves the claim of Proposition \ref{Symmetry_of_CMC_surfaces}.

Hence, let us assume that \eqref{The_centers_of_mass__symmetry} is true for some $\atime\in\interval*0*1$ and $\Hradius\in\interval{\Hradius_0}\infty$ and prove \eqref{The_centers_of_mass__symmetry_preserved}. We suppress the indices $\Hradius$ and $\atime$ in the following and denote the symmetric and the antisymmetric part of any function $h\in\Lp^2(\M)$ by 
\[ h_s := \frac{h+h\circ\graphF\circ\psi\circ\graphF^{-1}}2, \qquad h_a := \frac{h-h\circ\graphF\circ\psi\circ\graphF^{-1}}2, \]
respectively. By using the estimates on $f$ proven in Proposition \ref{Regularity_of_surfaces_in_asymptotically_flat_spaces}, we can compare $\M$ to $\sphere^2_\Hradius(0)$ and see that there is a complete $\Lp^2$-orthonormal system $\{\eflap_{i,j}\}_{j=-i...i,\,i\in\N_0}$ of eigenfunctions of the (negative) Laplacian such that
\[ \vert(\eflap_{i,j})_a\vert \le \frac C{\Hradius^{\frac32+\outve}} \text{ if $i$ is even }\quad\text{and}\quad\vert(\eflap_{i,j})_s\vert \le \frac C{\Hradius^{\frac32+\outve}} \text{ if $i$ is odd.} \]
Due to Proposition \ref{Regularity_of_surfaces_in_asymptotically_flat_spaces}, we know $\vert\zFundtrf*\vert^2\le\nicefrac C{\Hradius^{3+\outve}}$ and derive
\[ \vert(\jacobiext*\eflap_{i,j})_a\vert \le \frac C{\Hradius^{4+\outve}},\quad \vert(\jacobiext*\eflap_{i+1,j'})_s\vert \le \frac C{\Hradius^{4+\outve}} \qquad\forall\,i\text{ even}. \]
In particular, Proposition \ref{Stability} implies
\[ \Vert\deform\rnu_a\Vert_{\Hk^2(\M)} \le C\,\Hradius^{1-\outve}, \quad\text{if } \Vert(\jacobiext*\rnu)_a\Vert_{\Lp^2(\M)} \le \frac C{\Hradius^{1+\outve}},\qquad\Vert\jacobiext*\rnu\Vert_{\Lp^2(\M)} \le \frac C{\Hradius^{\frac12}}. \]
Furthermore, we know by the proof of Proposition \ref{I=J} that
\[ \jacobiext*\rnu = {-}\outmomden*(\nu) + \div\,\outzFund_\nu - \trtr\zFund\outzFund, \]
where $2\,\outzFund*=\schwarzoutg*-\outg*$ and $\outmomden*:=\outdiv(\outtr\,\outzFund-\outzFund*)$. Hence, the assumptions on the symmetry of $\M$ and $\outg*$ imply $\Vert\deform\rnu_a\Vert_{\Hk^2(\M)} \le C\,\Hradius^{1-\outve}$. Thus, we know that the leaves are (asymptotically) symmetric if
\[ \Vert\trans\rnu\Vert_{\Lp^2(\M)} \le C\,\Hradius. \]
Using the same argument as in the proof of Lemma \ref{I=J}, this is the case if and only if
\[ \vert\int_{\M}({-}\outmomden*(\nu) + \div\outzFund_\nu - \trtr\zFund\outzFund)\;\nu_i\d\mug\vert \le \frac C\Hradius. \]
Furthermore after integrating by parts and using $\vert\zFundtrf\vert\le\nicefrac C{\Hradius^{\frac32+\outve}}$ as well as the anti-symmetry of $\nu$ and $\zFund*$ resulting from the symmetry of $f$, we can use Lemma \ref{Estimating_the_center__implicit__strong_version} to conclude
\[ \vert I_i(\M) - I_i(\sphere^2_\Hradius(0))\vert \le \frac C{\Hradius^{1+\outve}}. \]
Here, we used the abbreviated form
\[ I_i(\M') := \int_{\M'}({-}\outmomden*(N) + \div\,\outzFund_N - \trtr\zFund\outzFund)N_i\d\mug \]
for any \pagebreak[1]closed hypersurface $\M'\hookrightarrow\outM\setminus\outsymbol L$ and its outer unit normal $N$. Thus, the combination of the Propositions \ref{Stability} and \ref{Movement_of_the_spheres_by_the_lapse_function} implies
\[ \vert\partial[\atime]@{\centerz[\Hradius]<\atime>} - \frac\Hradius{8\pi m}\,I_i(\sphere^2_\Hradius(0))\vert \le \frac C{\Hradius^{\outve}}\,(\vert\partial[\atime]@{\centerz[\Hradius]<\atime>}\vert+1). \]
Finally, we see that $(8\pi m)^{-1}\,\Hradius\;I_i(\sphere^2_\Hradius(0))$ does (in highest order) not depend on $\atime$ and that its limit for $\Hradius\to\infty$ is the ADM-center of mass. So, the existence of the ADM-center of mass implies that $\vert\partial*_\atime\centerz[\Hradius]<\atime>-\outcenterz_{\text{ADM}}\vert\to0$ for $\Hradius\to\infty$, which implies $\vert\centerz[\Hradius]<\atime>-\atime\,\outcenterz_{\text{ADM}}\vert\to0$ for $\Hradius\to\infty$. Hence, the CMC-center of mass exists and coincides with the ADM-center of mass if the ADM-center of mass exists. On the other hand, if the CMC-center of mass exists, i.\,e.\ $\vert\centerz[\Hradius]<\atime>-\outcenterz_{\text{CMC}}\vert\to0$, then the above argument proves that $\vert\partial*_\atime{\centerz[\Hradius]<\atime>}-\outcenterz_{\text{CMC}}\vert\to0$ for $\Hradius\to\infty$ as $\partial*_\atime{\centerz[\Hradius]<\atime>}$ is constant in $\atime$ in highest order. As the limit ($\Hradius\to\infty$) of $\nicefrac\Hradius{8\pi m}\;I_i(\sphere^2_\Hradius(0))$ is the ADM-center of mass, this proves that the ADM-center of mass exists and that it coincides with the CMC-center of mass if the CMC-center of mass exists. Thus, the two centers exist and coincides if one of them exists.
\end{proof}
\begin{remark}[Change of coordinates]\label{Change_of_coordinates}
Assume that $(\outM,\outg*,\outx)$ is $\Ck^2_{\frac12+\ve}$-as\-ymp\-to\-ti\-cally flat and satisfies the $\Ck^2_{\frac32+\outve}$-Regge-Teitel\-boim conditions. By repeating the proof of Lemma \ref{Ric(nu,e_i)-integrals} and Lemma \ref{Ric(nu,nu)nu_i-integrals}, we see that the results of both hold one order higher (in this setting), i.\,e.\ there are constants $\eta_i=\Cof{\eta_i}[\outM]\in\R$ such that
\[ \vert \int_{\M<\Hradius>}\outric*(\nu,\outsymbol e_i) - \frac\outsc2\,\nu_i \d\mug \vert \le \frac C{\Hradius^{1+\outve}}, \quad
	\vert \frac{\eta_i}\Hradius-\int_{\M<\Hradius>}(\outric*(\nu,\nu) - \frac\outsc2)\,\outx_i \d\mug \vert \le \frac C{\Hradius^{1+\outve}}, \]
where we again used Proposition \ref{Ordinary_differential__asymptotic__equations}. Repeating the proof of the foliation property in Theorem \ref{Uniqueness_of_the_CMC_surfaces}, we see that $\Vert\trans{\rnu}-\nicefrac{\eta\,\nu}\Hradius\Vert_{\Lp^\infty(\M)} \le \nicefrac C{\Hradius^{1+\outve}}$ and we already know $\Vert\deform{\rnu}\Vert_{\Lp^\infty(\M)}\le\nicefrac C{\Hradius^{\frac12+\ve}}$. Both inequalities are geometric ones, i.\,e.\ do not depend on the chosen coordinate system. As the center of mass is well-defined, combining these inequalities with Proposition \ref{Movement_of_the_spheres_by_the_lapse_function} implies $\eta_i=0$ as $\centerz<\Hradius>$ converges to the CMC-center of mass (i.,e.\ $\spartial[\Hradius]@{\centerz<\Hradius>}$ has finite integral).

%Thus, if we look at another $\Ck^2_{\frac12+\ve}$-asymptotically flat coordinate system $\outsymbol y$ only satisfying the $\Ck^2_{1+\outve}$-Regge-Teitel\-boim conditions, then Proposition \ref{Movement_of_the_spheres_by_the_lapse_function} implies that $\vert\spartial[\Hradius]@{\centerz<\Hradius>}\vert\le\nicefrac C{\Hradius^{1+\outve}}$. Thus, we know that the derivative of $\centerz<\Hradius>$ is integrable and therefore the CMC-center of mass is well-defined with respect to $\outsymbol y$. However, the ADM-center of mass \emph{can} be \emph{not} well-defined if the $\Ck^2_{1+\outve}$-Regge-Teitel\-boim conditions are not satisfied.

%In particular, there cannot exist a coordinate system satisfying the $\Ck^2_{\frac32+\outve}$-Regge-Teitel\-boim conditions for a Riemannian manifold for which a $\Ck^2_{\frac12+\ve}$-flat coordinate system exists which satisfies the $\Ck^2_{1+\outve}$-Regge-Teitel\-boim condition and with respect to which the (coordinate) CMC-center of mass is not well-defined.

%Note that the above does \emph{not} imply that the existence of the (coordinate) CMC-center of mass with respect to a $\Ck^2_{\frac12+\ve}$-asymptotically flat coordinate system $\outx$ only satisfying the $\Ck^2_{1+\outve}$-Regge-Teitel\-boim conditions implies the existence with respect to an other $\Ck^2_{\frac12+\ve}$-asymptotically flat coordinate system satisfying the $\Ck^2_{1+\outve}$-Regge-Teitel\-boim conditions.
\end{remark}

\section{Alternative assumptions}\label{Weaker_decay_assumptions}
In this small section, we explain alternative assumptions on $\outg$ and $\outsc$. Although, we state the assumption in their most natural form, we can also choose the assumption on $\outg$ independently from the on~$\outsc$, e.\,g.\ we can also assume pointwise assumptions on $\outsc$ and $\Wkp^{3,p}$-assumptions on $\outg*$.
Instead of~\eqref{Decay_assumptions_g}, we can assume
\begin{enumerate}
\item weaker pointwise assumptions\vspace{-.5em}
\begin{equation*} \qquad
 \vert\outg_{ij}-\eukoutg_{ij}\vert + \rad\,\vert\outlevi_{ij}^k \vert + \rad^2\,\vert\outric_{ij} \vert + \rad^2\vert\outsc\vert \le \frac{\oc(\rad)}{\rad^{\frac12}},\vspace{-.25em}
\end{equation*}
where $\oc\in\Ck(\interval*0\infty)$ is a function with $\oc(\rad)\to0$ as $\rad\to\infty$.
\item Sobolev assumptions\footnote{Note that there is slight mistake in \cite{nerz2015CMCfoliation}, where only $\outsc\in\Lp^1=\Wkp^{0,1}_{{-}3}(\outM)$ instead of $\outsc\in\Wkp^{1,1}_{{-}3}(\outM)$ was assumed.}\vspace{-.25em}
\[ \qquad\outg_{ij}-\eukoutg_{ij} \in\Wkp^{3,p}_{\nicefrac{{-}1}2}(\outM), \qquad\outsc\in\Wkp^{1,1}_{{-}3}(\outM),\qquad\quad p>2,\vspace{-.25em} \]
where we identified $\outM$ and $\outM\setminus\outsymbol L$ for notation convenience and used Bartnik's definition of weighted Sobolev spaces \cite[Def.~1.1]{bartnik1986mass}, i.\,e.\vspace{-.25em}
\begin{align*}&&
 \Vert T\Vert_{\Wkp^{0,p}_\eta(\outM)} :={}& \Vert \rad^{-\eta-\frac3p}\,\vert T\vert_{\outg*} \Vert_{\Lp^p(\outM)},\\ &&\vspace{-.25em}
   \Vert T\Vert_{\Wkp^{k+1,p}_\eta(\outM)} :={}& \Vert T\Vert_{\Wkp^{0,p}_\eta(\outM)} + \Vert\outlevi*T\Vert_{\Wkp^{k,p}_{\eta-1}(\outM)} \nopagebreak\vspace{-.25em}
\end{align*}
for any function (or tensor field) $T$ and constants $\eta\in\R$, $k\in\N_0$.
\end{enumerate}
If one of these assumptions is satisfied, then the CMC-foliation exists (as stated in Theorem~\ref{existence_of_the_CMC-foliation}), the CMC-leaves are unique within $\M\in\mathcal A^{0,\eta}_\Aradius(\c_0,\c_1)$ (as stated in Theorem~\ref{Uniqueness_of_the_CMC_surfaces} for $\outve=0$), and the three smallest eigenvalues $\ewjac<\Hradius>_i$ of the stability operator satisfy
\[ \vert\ewjac<\Hradius>_i-\frac{6\,\mass}{\Hradius^3} \vert \le \mathcal o(\Hradius^{{-}3}) \qquad\forall\,i\in\{1,2,3\} \]
(as stated in Theorems~\ref{Stability_of_the_foliation} for $\mathcal o(\Hradius^{{-}3})$ instead of $\frac C{\Hradius^{3+\outve}}$), where $\mathcal o(\Hradius^{{-}3})$ denotes a function with $\Hradius^3\,\mathcal o(\Hradius^{{-}3})\to 0$ for $\Hradius\to\infty$ which depends on $\oc$ and on $\Vert\outg-\eukoutg\Vert_{\Wkp^{3,p}_{\nicefrac{{-}1}2}(\outM)}$ in the first and second setting, respectively.

\subsection{Assuming non-negativity of \texorpdfstring{$\outsc$}{the scalar curvature}}
If $\outg*$ satisfies one of the assumption explained above (or the one in ~\eqref{Decay_assumptions_g}), but $\outsc$ satisfies only $\outsc\in\Lp^1$, then the proofs of this paper cannot be applied---more precisely the central inequalities~\eqref{Eigenvalues_of_the_Laplacian__t}, \eqref{Eigenvalues_of_the_Laplacian__d} and \eqref{Stability__t}---\eqref{Stability__ewjac} do not hold anymore. However, the proof of Lemma~\eqref{Eigenvalues_of_the_Laplacian} can be applied to deduce (in the notation of this lemma)
\[ \vert\ewlap_i - (\frac2{\Hradius^2} - \frac{6\,\mHaw}{\Hradius^3} + \int_{\M}(\outric*(\nu,\nu)-\frac12\outsc)\,f_i^2 \d\mug)\vert \le \frac C{\Hradius^{3+\ve}} \qquad\forall\,i\in\{1,2,3\} \]
and
\[ \ewlap_k > \frac 5{\Hradius^2} \quad\forall\,k > 3, \qquad\ \vert\int_{\M}(\outric*(\nu,\nu)-\frac12\outsc)\,f_i\,f_j \d\mug \vert \le \frac C{\Hradius^{3+\ve}} \quad\forall\,i\neq j\in\{1,2,3\} \]
instead of~\eqref{Eigenvalues_of_the_Laplacian__t} and~\eqref{Eigenvalues_of_the_Laplacian__d}, respectively, i.\,e.\ we have to replace $\outric$ by $\outric-\frac12\outsc\outg$. Therefore, if we assume that $\outsc$ is additionally non-negative, more precisely ${-}\oc\,\rad^{{-}3-\outve}\le\outsc$ and $\outsc\in\Lp^1(\outM)$, then we can apply the proof of Proposition~\ref{Stability} to get (in the notation of this proposition)
\[  \int_{\M} (\jacobiext*\trans g)\,\trans h\d\mug \ge \frac{6\mHaw}{\Hradius^3}\int_{\M}\trans g\,\trans h\d\mug - \frac C{\Hradius^{3+\ve}}\Vert\trans g\Vert_{\Lp^2(\M)}\,\Vert\trans h\Vert_{\Lp^2(\M)} \]
and
\[ \vert\,\ewjac\vert \ge \frac3{2\Hradius^2} \qquad\text{or}\qquad\Vert\,\deform{\efjac}\Vert_{\Hk^2(\M)} \le \frac C{\Hradius^{\frac12+\outve}}\Vert\,\efjac\Vert_{\Hk^2(\M)},\quad \ewjac\ge\frac{6\mHaw}{\Hradius^3}-\frac C{\Hradius^{3+\outve}} \]
instead of~\eqref{Stability__t} and~\eqref{Stability__ewjac}, respectively. Inequality~\eqref{Stability__d} remains valid in its original form. The famous positive mass theorem by Schoen-Yau implies $\mass>0$, \cite{schoen1981proof} and see \cite{schoen1989variational} for the version used here. Therefore, we also know $\mHaw>0$ (see Appendix~\ref{ricci_integrals_and_the_mass}). All in all, the arguments explained in Section~\ref{existence_of_the_CMC-foliation} therefore prove the following result for the situation closer to the one of the positive mass theorem:

\begin{theorem}[Assuming non-negativity of \texorpdfstring{$\outsc$}{the scalar curvature}]
Let $\oc$ be a constant and $(\outM,\outg*)$ be a three-dimensional and non-flat Riemannian manifold with integrable scalar curvature and assume there exists a smooth chart $\outx:\outM\setminus\outsymbol L\to\R^3$ of $\outM$ outside a compact set $\overline L\subseteq\outM$. If there exists a function $\oc'$ with $\frac{\oc'}{\rad^3}\in\Lp^1(\outM\setminus\outsymbol L)$ with\vspace{-.25em}
\[ \vert\outg_{ij}-\eukoutg_{ij}\vert + \rad\,\vert\outlevi_{ij}^k \vert + \rad^2\,\vert\outric_{ij} \vert \le \frac{\oc'}{\rad^{\frac12}},\quad
	{-}\frac{\oc'}{\rad^3} \le \outsc, \quad
		\oc'(\rad)\to 0 \text{ as } \rad\to\infty\vspace{-.25em} \]
or there exist constants $p>2$ and $\oc$ with $\Vert\outg_{ij}-\eukoutg_{ij}\Vert_{\Wkp^{3,p}_{\nicefrac{{-}1}2}(\outM\setminus\overline L)}\le \oc$\vspace{-.1em}, and \hbox{{\normalfont(}in} both \hbox{cases\hspace{.05em}\normalfont)} $\mass\neq0$,
then there exist a constant $\Hradius_0=\Cof{\Hradius_0}[\mass][\oc][\oc'][\int\outsc\d\outmug]$ and a smooth family of unique CMC-surfaces $\{\M<\Hradius>\}_{\Hradius>\Hradius_0}$ foliating $\outM$ outside of a compact set---as stated in Theorems~\ref{Existence_of_a_CMC-foliation} and~\ref{Uniqueness_of_the_CMC_surfaces} \hbox{{\normalfont(}for} \hbox{$\outve=0$\hspace{.05em}\normalfont)}. Furthermore, the eigenvalues $\ewjac<\Hradius>_i$ of the {\normalfont(}negative{\hspace{.05em}\normalfont)} stability operator of the CMC-leave $\M<\Hradius>$ with mean curvature $\H={-}\frac2\Hradius$ satisfy\vspace{-.25em}
\[ \ewjac<\Hradius>_0 < {-}\frac3{2\Hradius^2}, \quad
	\ewjac<\Hradius>_i\ge\frac{6\,\mass}{\Hradius^3} - \mathcal o(\Hradius^{{-}3}), \quad
	\ewjac<\Hradius>_j > \frac5{\Hradius^2}\vspace{-.25em},\qquad
	\forall\,i\in\{1,2,3\},\,j\ge4, \]
where $\mathcal o$ is a function depending on $\mass$, $\oc$, $\oc'$, and $\int\outsc\d\outmug$ with $\Hradius^3\,\mathcal o(\Hradius^{{-}3})\to 0$ for $\Hradius\to\infty$.
\end{theorem}

\appendix\normalsize
\section{Ricci-integrals and the mass}\label{ricci_integrals_and_the_mass}
As explained in Section \ref{Assumptions_and_notation}, \eqref{Definition_of_mass} gives a characterization of the ADM-mass of an asymptotically flat manifold. It is well-known that this mass is well-defined, see e.\,g.\ \cite{schoen1988existence,Corvino__On_the_center_of_mass_of_isolated_systems,miao2013evaluation}. However, in order to recall the convergence rate, we will repeat the proof nevertheless in Lemma \ref{Ric(nu,nu)-integrals}. Furthermore, we want to recall that this mass is the limit of Hawking masses. To see this, we use the Gau\ss\ equation in the definition of $\mass$ to see that for any sufficiently large $\rradius$ and any $\Ck^2_{\frac12+\ve}$-asymptotically flat Riemannian manifold $(\outM,\outg*,\outx)$
\[ \vert\frac\rradius{8\pi}\int_{\sphere^2_\rradius}\frac\outsc2-\outric*(\nu,\nu)\d\mug - \mHaw<\rradius>\vert
	\le \vert\frac{\sqrt{\volume{\M}}}{32\pi^{\frac32}}\int_{\sphere^2_\rradius}\sc-\frac{\H^2}2+\trtr{\zFundtrf<\rradius>}{\zFundtrf<\rradius>} \d\mug - \mHaw<\rradius>\vert + \frac C{\rradius^{\outve}} \]
holds, where $\nu$ and $\mHaw<\rradius>$ denote the unit normal and Hawking mass of $\sphere^2_\rradius:=\outx^{-1}(\sphere^2_\rradius(0))\hookrightarrow\outM$, respectively. Using the decay assumption on $\outg*-\eukoutg*$ and the Gau\ss-Bonnet theorem, this implies
\begin{equation*} \mass = \lim_{\rradius\to\infty} \mHaw(\sphere^2_\rradius(0)). \tag{\ref{Definition_of_mass}'} \end{equation*}
It should be noted that this definition of mass is a purely geometric definition, which can be seen by replacing $\sphere^2_\rradius$ by the CMC-leaf $\M<\Hradius>$ and using Proposition \ref{Ric(nu,nu)-integrals}.\smallskip\pagebreak[3]

Let use recall the proof that the mass is well-defined if the scalar curvature is integrable.
\begin{proposition}[Mass is well-defined or \texorpdfstring{$\outric*(\nu,\nu)$}{Ric(nu,nu)}-integrals]\label{Ric(nu,nu)-integrals}
Let $(\outM,\outg*,\outx)$ be a $\Ck^2_{\frac12+\ve}$-asymptotically flat Riemannian manifold. The mass 
\begin{equation*} \mass := \lim_{\rradius\to\infty} \frac{{-}\rradius}{8\pi} \int_{\sphere^2_\rradius} \frac\outsc2 - \outric*(\nu,\nu) \d\mug \tag{\ref{Definition_of_mass}}\end{equation*}
of $(\outM,\outg*)$ is well-defined, where $\sphere^2_{\rradius}:=\outx^{-1}(\sphere^2_{\rradius}(0))$. Assume $\M\hookrightarrow\outM\setminus\outsymbol L$ is a closed hypersurface enclosing $\outsymbol L$, i.\,e.\ $(\outM\setminus\outsymbol L)\setminus\M$ consists of two connected subsets of $\outM$ and $\outsymbol L$ is within the relatively compact one. For any constant $\c\ge0$ there is a constant $C=\Cof[\ve][\outc][\c]$ such that the existence of a vector $\centerz\in\R^3$ with
\[ \vert\centerz\,\vert \le \c\,\rradius,\qquad
	\max_{\M}\hspace{.05em}\vert\hspace{.05em}\nu-\frac{\outx-\centerz}\rradius\hspace{.05em}\vert \le \frac{\c}{\rradius^{\frac{1+\outve}2}},\qquad
	\volume{\M} \le \c\,\rradius^2, \]
implies
\[ \vert \mass + \frac{\rradius}{8\pi} \int_{\M}(\outric*(\nu,\nu) - \frac\outsc2) \d\mug \vert \le \frac C{\rradius^{\outve}} \]
where $\rradius:=\min_{\M}\rad$, $\nu$ and $\mug$ denote the minimal distance to the coordinate origin, the outer unit normal and the surface measure of $\M\hookrightarrow\outM$, respectively.
\end{proposition}
\begin{proof}
First, let us assume that $\centerz=0$ and identify $\outM\setminus\outsymbol L$ with $\R^3\setminus B_1(0)$. By the second Bianchi identity, we know that $\outric*-\frac12\,\outsc\,\outg*$ is divergence-free. Denoting with $U\subseteq\R^3$ the set \lq enclosed\rq\ by $\outx(\M)$, we conclude
\begin{align*}\hspace{4em}&\hspace{-4em}
 \vert S\int_{\sphere^2_S(0)} \outric*(\nu,\nu)- \frac{\outsc}2 \d\mug - \rradius\int_{\M} \outric*(\nu,\nu) - \frac{\outsc}2\d\mug\vert \\
	\le{}& \vert \int_{\sphere^2_S(0)} \outric*(\nu,\outx)- \frac{\outsc}2\outg*(\nu,\outx) \d\mug - \int_{\M} \outric*(\nu,\outx)- \frac{\outsc}2\outg*(\nu,\outx)\d\mug\vert + \frac C{\rradius^{\outve}} \\
	\le{}& \vert\int_{U\setminus B_S(0)} (\outric*-\frac\outsc2\;\outg*)(\outlevi*x) \d\outmug \vert + \frac C{\rradius^{\outve}}
	\le \int_{U\setminus B_S(0)} \frac{\vert\outsc\vert}2 \d\outmug + \frac C{\rradius^{\outve}}
\end{align*}
for any $S>\rradius$, where $\outmug$ denotes the volume measure of $\outg*$ $\nu$ and $\mu$ denote the corresponding normal and surface measure of $\M\hookrightarrow\outM$ or $\sphere^2_S(0)\hookrightarrow\outM$, respectively. By the assumption on $\outsc$, this implies that the mass is well-defined and the claim for the special case of $\centerz=0$. Thus, the proposition is proven if
\[ \vert\int_{\M} (\outric*-\frac\outsc2\;\outg*) (\nu,\centerz\,) \d\mug \vert \le \frac C{\rradius^{\outve}} \]
holds for any $\M$ satisfying the above assumptions. However, we prove this in Lemma \ref{Ric(nu,e_i)-integrals}.\pagebreak[2]
\end{proof}
As we saw in the proof of the last proposition, we need additionally to control simpler $\outric*$-integrals to get the estimate on the mass integral for a large class of surfaces. Furthermore, we will need this technical lemma again later.
\begin{lemma}[\texorpdfstring{$\outric*(\nu,\outsymbol e_i)$}{Ric(nu,e_i)}-integrals]\label{Ric(nu,e_i)-integrals}
Let $(\outM,\outg*,\outx)$ be a $\Ck^2_{\frac12+\ve}$-asymptotically flat Riemannian manifold. Assume $\M\hookrightarrow\outM\setminus\outsymbol L$ is a closed hypersurface enclosing $\outsymbol L$, i.\,e.\ $\outM\setminus\M$ consists of two connected subsets of $\outM$ and $\outsymbol L$ is within the compact one. There is a constant $C=\Cof[\ve][\outc]$ such that
\[ \vert \int_{\M}\outric*(\nu,\outsymbol e_i) - \frac\outsc2\,\nu_i \d\mug \vert \le \frac C{\rradius^{1+\outve}} \]
where $\rradius:=\min_{\M}\rad$, $\nu$ and $\mug$ denote the minimal distance to the coordinate origin, the outer unit normal and the surface measure of $\M\hookrightarrow\outM$, respectively.
\end{lemma}
\begin{proof}
By the second Bianchi identity, we know that $\outric*-\frac12\,\outsc\,\outg*$ is divergence free. Denoting with $U\subseteq\R^3$ the set enclosed by $\outx(\M)$, we conclude
\begin{align*}\hspace{4em}&\hspace{-4em}
 \vert\int_{\sphere^2_S(0)}\outric*(\nu,\outsymbol e_i) - \frac\outsc2\,\nu_i \d\mug - \int_{\M}\outric*(\nu,\outsymbol e_i) - \frac\outsc2\,\nu_i \d\mug \vert \\
	={}& \vert\int_{U\setminus B_S(0)}(\outric*-\frac\outsc2\;\outg*)(\outlevi*\outsymbol e_i) \d\outmug \vert \\
	\le{}& \int_{\rradius}^S\int_{\sphere^2_r(0)} \vert\outric*-\frac\outsc2\;\outg*\vert_{\outg*}\,\vert\outlevi*\outsymbol e_i\vert \d\mug \d r \le \frac C{\rradius^{1+\outve}}
\end{align*}
for every $S>\rradius$, where $\outmug$ denotes the volume measure of $\outg*$ $\nu$ and $\mu$ denote the corresponding normal and surface measure of $\M\hookrightarrow\outM$ or $\sphere^2_r(0)\hookrightarrow\outM$, respectively. This proves the claim, as the assumption on the curvature $\outric*$ implies
\[ \vert\int_{\sphere^2_R(0)}\outric*(\nu,\outsymbol e_i) - \frac\outsc2\,\nu_i \d\mug \vert \le \frac C{R^{\frac12+\outve}} \xrightarrow{R\to\infty} 0. \pagebreak[3]\qedhere \]%
\end{proof}%
In the proof of Theorem \ref{Existence_of_a_CMC-foliation}, we had to prove that the constructed CMC-cover is a foliation. To do so, we used that the integral $\vert\int\outric*(\nu,\nu)\,\nu_i \d\mug \vert$ is sufficiently small. This is done in the following Lemma which (for this reason) we call \lq foliation lemma\rq.
\begin{lemma}[Foliation lemma or \texorpdfstring{$\outric*(\nu,\nu)\,\nu_i$}{Ric(nu,nu) nu_i}-integrals]\label{Ric(nu,nu)nu_i-integrals}
Assume that $\M\hookrightarrow(\outM,\outg,\outx)$ satisfies the assumptions of Proposition \ref{Ric(nu,nu)-integrals}. There is a constant $C=\Cof[\outve][\oc][\c]$ such that
\[ \max_{\M} \vert\outx - \centerz\,\vert \le \rradius + \c\,\rradius^{\frac12-\outve} \]implies
\[ \vert\int(\outric*(\nu,\nu) - \frac\outsc2)\,\outx_i \d\mug \vert \le \frac C{\rradius^{\outve}}, \quad
	\vert \frac{m\,\centerz_i}\rradius - \frac{\rradius}{8\pi}\int(\outric*(\nu,\nu) - \frac\outsc2)\,\nu_i \d\mug \vert \le \frac C{\rradius^{\outve}}. \]
\end{lemma}
\begin{proof}
It is sufficient to prove the first inequality due to Proposition \ref{Ric(nu,nu)-integrals}. Let
\gdef\radz{|\outx-\centerz\,|}
$R:=\max_{\M}\radz$ denote the distance from the center of $\M$ and its maximum on $\M$. Again using the divergence theorem, we get
\begin{align*}\hspace{4em}&\hspace{-4em}
 \vert\int_{\sphere^2_\rradius(\centerz\,)}(\outric*(\nu,\nu) - \frac\outsc2)\,\outx_i \d\mug - \int_{\M}(\outric*(\nu,\nu) - \frac\outsc2)\,\outx_i \d\mug \vert \\
	={}& \vert\int_U(\outric*-\frac{\outsc}2\,\outg*)(\outlevi*(\frac{\outx-\centerz}{\radz}\,\outx))\d\outmug \vert + \frac C{\rradius^{\outve}} \\
	\le{}& \int_\rradius^{\rradius+c\rradius^{\frac12-\outve}}\int_{\sphere^2_r(\centerz\,)} \frac C{r^{\frac52+\outve}} \d\mug \d r + \frac C{\rradius^{\outve}} \le \frac C{\rradius^{\outve}}
\end{align*}
where we used the same notation as in Lemma \ref{Ric(nu,e_i)-integrals}. Thus, it is sufficient to prove the claim for $\M=\sphere^2_\rradius(\centerz\,)$.

Define the function $f$ for any sufficiently large $\rradius\gg1$ by
\[ f(\rradius) := \int_{\sphere^2_\rradius(\centerz\,)} (\outric*-\frac{\outsc}2\,\outg*)(\nu,\frac{\outx-\centerz\,}{\rradius}) \frac{\outx_i}\rradius \d\mug \]
and note that a priori this function seems only to be continuous.%
Now, we prove that this function is differentiable and (asymptotically) satisfies an ordinary differential equation. Then, solving this equation will prove the claim.\pagebreak[2]

Again using the divergence theorem, we see\vspace{-.25em}
\[ f(R') - f(R'') - \int_{R''}^{R'}\int_{\sphere^2_\rradius(0)}(\outric*-\frac{\outsc}2\,\outg*)(\outlevi*(\frac{(\outx-\centerz\,)\,\outx_i}{\radz^2})) \d\mug \d\rradius = \int_{R''}^{R'} \text{err}_\rradius \d\rradius, \]
for any $R'>R''$, where $\text{err}_\rradius$ is a error term with $\vert\text{err}_\rradius\vert\le\nicefrac C{\rradius^{2+\outve}}$. Note that we do not evaluate integral on the right hand side. \pagebreak[3]Therefore, $f$ is differentiable and conclude using $\vert\outlevi*(\outx-\centerz\,)-\id\vert\le\nicefrac C{\rad^{\frac12+\ve}}$ that
\begin{align*}
 \vert\partial[\rradius]@{f(\rradius)} + \frac2\rradius\,f(\rradius) \vert
	\le{}& \vert \int_{\sphere^2_\rradius(\centerz\,)}(\outric*-\frac{\outsc}2\hspace{.05em}\outg*)(\outx-\centerz,\outlevi*(\frac{\outx_i}{\radz^2})+\frac{2\outx_i\hspace{.05em}(\outx-\centerz\,)}{\radz^4}) \d\mug \vert \\
		& + \frac1\rradius\int\vert\outsc\vert\d\mug + \frac C{\rradius^{2+\outve}} \\
	\le{}& \vert\int_{\sphere^2_\rradius(\centerz\,)} (\outric*-\frac{\outsc}2\,\outg*)(\nu,\frac{\outsymbol e_i}\rradius)\vert + \frac1\rradius\int\vert\outsc\vert\d\mug + \frac C{\rradius^{2+\outve}}.
\end{align*}
Using Lemma \ref{Ric(nu,e_i)-integrals}, we conclude
\[ \vert\partial[\rradius]@{f(\rradius)} + \frac{2\,f(\rradius)}\rradius \vert \le \frac C{\rradius^{2+\outve}}. \]
Solving this ordinary differential (asymptotic) equation (this is done in more detail in Proposition \ref{Ordinary_differential__asymptotic__equations}), we conclude for some $\eta\in\R$ that
\[ \vert f(\rradius) - \eta\vert \le \frac1{\rradius^{1+\outve}}. \]
We conclude the claim by using $f(\rradius)\to0$ for $\rradius\to\infty$.\pagebreak[3]
\end{proof}

\section[Integrals for the centers]{Integrals for the centers of the leaves and of the mass}
In the proof of Lemma \ref{I=J}, we control the derivative of the Euclidean coordinate center with respect to the radius by the integral
\[ I_i(\M) := \int_{\M<\Hradius>} \Hradius\,\outmomden(\nu<\Hradius>)\,\nu_i + \zFund*(\nu<\Hradius>,\outsymbol e_i-\nu<\Hradius>_i\,\nu<\Hradius>) - \troutzFund\,\nu<\Hradius>_i \d\mug \]
and conclude that the the CMC-surfaces exist if this term is of order $\Hradius^{1-\outve}$. Here, $2\outzFund*=\schwarzoutg*-\outg*$ and $\outmomden:=\outdiv(\outtr\,\outzFund-\outzFund*)$ are artificial quantities and $\M<\Hradius>$ is a surface of constant mean curvature $\nicefrac{{-}2}\Hradius$ with respect to $\outg[\atime]*:=\schwarzoutg+\atime\,(\outg*-\schwarzoutg*)$ satisfying the assumptions of Proposition \ref{Regularity_of_surfaces_in_asymptotically_flat_spaces}. Furthermore in the proof of Theorem \ref{The_centers_of_mass}, we need that $I_i(\M) - I_i(\sphere^2_\Hradius(0))$ is of order $\Hradius^{1+\ve}$ if $\M$ is additionally an asymptotically symmetric surface with center $\centerz$ of order $\Hradius^0$.

To show that $I_i(\M)$ decays as explained, we use the trick used throughout the literature to prove that the ADM-center of mass is well-defined if $(\outM,\outg*,\outx)$ satisfies the $\Ck^2_{\frac32+\outve}$-Regge-Teitel\-boim conditions. Let us begin by proving that $I_i(\M)$ is of order $\Hradius^{1-\outve}$.
\begin{lemma}[Estimating the center (implicitly)]\label{Estimating_the_center__implicit}
Let $(\outM,\outg*,\outx)$ be a three-di\-men\-sio\-nal $\Ck^2_{\frac12+\ve}$-asymptotically flat Riemannian manifold. Assume that $(\M,\g*)\hookrightarrow(\outM,\outg*)$ satisfies the assumptions of Proposition \ref{Regularity_of_surfaces_in_asymptotically_flat_spaces} with $\Aradius>\Aradius_0$. There is a constant $C=\Cof[\outve][\oc][\eta][\c_1][\c_2]$ such that
\[ \vert\int_{\M} \Aradius\,\outmomden(\nu)\,\nu_i + \zFund*(\nu,\outsymbol e_i-\nu_i\,\nu) - \troutzFund\,\nu_i \d\mug\vert \le C\,\Aradius^{1-\outve}, \]
where $2\outzFund*=\schwarzoutg*-\outg*$ and $\outmomden:=\outdiv(\outtr\,\outzFund*\,\outg*-\outzFund*)$ are artificial quantities and $\nu$ and $\mu$ denote the outer unit normal and the surfaces measure of $\M\hookrightarrow(\outM,\outg*)$, respectively.
\end{lemma}
Due to the assumptions on $\outzFund$, i.\,e.\ the ones on $\outg$, and $\M$, we can replace $\Aradius$ by $\rad_0$ with $\rad_0(\outx):=\vert\outx-\centerz\,\vert$ and $\outzFund(\nu,\outsymbol e_i-\nu_i\nu)$ by $\rradius\,\outzFund(\nu,\levi*\nu_i)$ in the definition of $I_i$ and still get the same result. Equally, we can replace $\troutzFund$ by $\trtr\zFund\outzFund$.
\begin{proof}
Let us begin by noting
$\vert\outtr\,\outzFund* - \eukouttr\,\outzFund*\vert \le \nicefrac C{\rad^{1+\outve}}$
and that the corresponding inequalities hold for the derivatives (up to the second order). Thus, we can replace $\outmomden$ by $\outdiv(\eukouttr\,\outzFund*\,\eukoutg*-\outzFund*)$. As the assumptions on the metric imply $\vert\outlevi_\relax\vert\le\nicefrac C{\rad^{\frac32+\outve}}$, we can replace $\outmomden$ by $\eukoutdiv(\eukouttr\,\outzFund*\,\eukoutg*-\outzFund*)$, too. Further recalling that $\outtr\,\outzFund*=\outzFund*(\nu,\nu)+\tr*\outzFund$, it is sufficient to prove
\[ \vert\int_{\M} \eukoutdiv\,\momentum*(\nu)\,\outx_i - \momentum*(\nu,\outsymbol e_i) \d\eukmug\vert \le C\,\Aradius^{1-\outve}, \]
where $\eukmug$ denotes the surface measure with respect to the Euclidean metric $\eukoutg*$ and $\momentum*:=\eukouttr\,\outzFund*\,\eukoutg*-\outzFund*$. Using the divergence theorem, we see that for any $r>0$ and $R:=\max_{\M}\rad$
\begin{align*}\hspace{2.5em}&\hspace{-2.5em}
 \vert\int_{\M} \eukoutdiv\,\momentum*(\nu)\,\outx_i - \momentum*(\nu,\outsymbol e_i) \d\eukmug - \int_{\sphere^2_r(0)}\eukoutdiv\,\momentum*(\nu)\,\outx_i - \momentum*(\nu,\outsymbol e_i) \d\eukmug \vert \\
	\le{}& \vert\int_r^R \int_{\sphere^2_s(0)} \eukoutdiv(\eukoutdiv\,\momentum*\,\outx_i - \momentum*_i) \d\eukmug \d s \vert
	\le \int_r^R \vert\int_{\sphere^2_s(0)} \eukoutdiv(\eukoutdiv\,\momentum*)\,\outx_i \d\eukmug \vert\d s
\end{align*}
holds. Fixing $r\in\interval{R_0}R$ arbitrary, we conclude all in all that
\[ \vert\int_{\M} \outmomden(\nu) + \frac{\zFund*(\nu,\outsymbol e_i-\nu_i\,\nu) - \troutzFund\,\nu_i}\Aradius \d\mug\vert \le \int_r^R \vert\int_{\sphere^2_s(0)} \eukoutdiv(\eukoutdiv\,\momentum*)\,\frac{\outx_i}\Aradius \d\eukmug \vert\d s + \frac C{\Aradius^{\outve}}. \]
Thus, the claim is proven if $\vert\eukoutdiv(\eukoutdiv\,\momentum*)\vert\le\nicefrac C{\rad^{3+\outve}}$. 

In coordinates, we know
\[ \vert\outric_{ij} - \partial[\outx]_k@{\outlevi_{ij}^k} + \partial[\outx]_j@{\outlevi_{ki}^k} \vert = \vert\outlevi_{kl}^k\,\outlevi_{ji}^l - \outlevi_{ji}^l\,\outlevi_{ki}^l\vert \le \frac C{\rad^{3+\outve}} \]
and conclude by plugging in the characterization of $\outlevi_\relax$ by the derivatives of the metric, that
\[ \vert 2\outric_{ij} - \partial[\outx]_j@{(\eukoutdiv\outg)_i} - \partial[\outx]_i@{(\eukoutdiv\outg)_j} + \eukoutHess_{ij}(\eukouttr\,\outg*) + \eukoutlaplace\outg_{ij} \vert \le \frac C{\rad^{3+\outve}}. \]
Using $\eukoutdiv\eukoutg*\equiv0$, this implies by tracing
\[ \vert\eukoutdiv(\eukoutdiv\,\momentum*)\vert \le \vert\outsc\vert + \frac C{\rad^{3+\outve}} \le \frac C{\rad^{3+\outve}}. \]
As explained above, this proves the claim.\pagebreak[3]
\end{proof}

\begin{proposition}[Estimating the center (implicitly) -- strong version]\label{Estimating_the_center__implicit__strong_version}
Let $(\outM,\outg*,\outx)$ be a three-dimensional $\Ck^2_{\frac12+\ve}$-asymptotically flat Riemannian manifold satisfying the $\Ck^2_{1+\outve}$-asymptotic Regge-Teitel\-boim conditions and define $I_i(M)$ for any closed hypersurface $M\hookrightarrow\outM\setminus\outsymbol L$ by
\[ I_i(M) := \int_{\M} \rad\;\outmomden(\nu)\,\nu_i + \zFund*(\nu,\outsymbol e_i-\nu_i\,\nu) - \troutzFund\,\nu_i \d\mug, \]
where $2\outzFund*=\schwarzoutg*-\outg*$ and $\outmomden:=\outdiv(\outtr\,\outzFund*\,\outg*-\outzFund*)$ are artificial quantities and $\nu$, and $\mu$ denote the outer unit normal and the surfaces measure of $M\hookrightarrow(\outM,\outg*)$, respectively. Assume that $(\M,\g*)\hookrightarrow(\outM,\outg*)$ satisfies the assumptions of Proposition \ref{Regularity_of_surfaces_in_asymptotically_flat_spaces} with $\Aradius>\Aradius_0$ and that $\M$ is {\normalfont(}asymptotically{\normalfont)} symmetric, i.\,e.\ there is a function $f\in\Hk^2(\sphere^2_\Aradius(\centerz\,))$ such that $\M:=\graph f$ and
\[ \vert\centerz\,\vert \le \c_1, \qquad \Vert f - f\circ\varphi \Vert_{\Hk^2(\sphere^2_\Aradius(\centerz\,))} \le C\,\Aradius^{1-\outve}, \]
where $\varphi:\sphere^2_\Aradius(\centerz\,)\to\sphere^2_\Aradius(\centerz\,):x+\centerz\mapsto{-}x+\centerz$\pagebreak[1]. There is a constant $C=\Cof[\outve][\oc][\eta][\c_1][\c_2]$ such that
\[ \Big| I_i(\M) - I_i\big(\sphere^2_R(\centerz\,)\big)\Big| \le \frac C{\Aradius^{\outve}} \qquad\forall\,R>\Aradius>R_0,\;i\in\{1,2,3\}. \]
Furthermore, the ADM-center of mass $\outcenterz_{\text{ADM}}$ exists if and only if $(I_i(\sphere^2_r(0)))_{i=1}^3$ converges for $r\to\infty$ and then these limits coincide. If $(\outM,\outg*,\outx)$ satisfies the $\Ck^2_{\frac32+\outve}$-asymptotic Regge-Teitel\-boim conditions, then $I_i(\sphere^2_r(0))$ converges with $r\to\infty$ for any $i\in\{1,2,3\}$.
\end{proposition}
Due to the (symmetry) assumptions on $\outzFund$, i.\,e.\ the ones on $\outg$, and $\M$, we can replace $\Hradius$ by $\rradius:=\min_M\rad$ and $\outzFund(\nu,\outsymbol e_i-\nu_i\nu)$ by $\rradius\,\outzFund(\nu,\levi*\nu_i)$ in the definition of $I_i$ and still get the same (asymptotic) result. Equally, we can replace $\troutzFund$ by $\rradius\,\trtr\zFund\outzFund$.
\begin{proof}
Assume that the $\Ck^2_{q+\outve}$-Regge-Teitel\-boim conditions are satisfied for $q=1$ or $q=\frac32$. Calculating as in the proof of \ref{Estimating_the_center__implicit}, we see
\[ \vert\eukoutdiv(\eukoutdiv\,\momentum*) \vert \le \frac C{\rad^{3+\outve}}, \qquad
		\vert\eukoutdiv(\eukoutdiv\,\momentum*)(\outx) - \eukoutdiv(\eukoutdiv\,\momentum*)({-}\outx) \vert \le \frac C{\rad^{3+q+\outve}}. \]
Equally, we see that
\[ \vert (\outtr\,\outzFund-\eukouttr\,\outzFund)(\outx) - (\outtr\,\outzFund-\eukouttr\,\outzFund)({-}\outx) \vert \le \frac C{\rad^{1+q+\outve}}. \]
Thus, the same arguments as in the proof of \ref{Estimating_the_center__implicit} result in
\begin{align*}\hspace{2.5em}&\hspace{-2.5em}
 \vert\int_{\M} \eukoutdiv\,\momentum*(\nu)\,\outx_i - \momentum*(\nu,\outsymbol e_i) \d\eukmug - \int_{\sphere^2_r(0)}\eukoutdiv\,\momentum*(\nu)\,\outx_i - \momentum*(\nu,\outsymbol e_i) \d\eukmug \vert \\
	\le{}& \int_r^R \vert\int_{\sphere^2_s(0)} \eukoutdiv(\eukoutdiv\,\momentum*)\,\outx_i \d\eukmug \vert\d s \le \frac C{\rad^{q+\outve}},
\end{align*}
where we replaced $q$ by $1$ if $M\neq\sphere^2_\Aradius(0)$. This implies the claims except for the comparability with the ADM-center of mass, which we get by writing the integrand in coordinates.
\end{proof}

\section{Ordinary differential (asymptotic) equations}
We used that a function, which satisfies an ordinary differential equation asymptotically, i.\,e.
\[ \vert \partial[t]@{f(t)} + \frac{\delta\,f(t)}t - \frac \eta t \vert \le \frac\c{t^\ve} \]
is asymptotic to the corresponding solution, i.\,e.
\[ \vert f(t) - \frac{\eta'}{t^\delta} - \frac\eta\delta \vert \le \frac C{t^{\ve-1}} \]
for some constant $C=\Cof[\c][\delta]$ and $\eta'=\Cof{\eta'}[f]$. This is intuitively clear and the proof straightforward, but we prove it never the less for the sake of completeness.
\begin{proposition}[Ordinary differential (asymptotic) equations]\label{Ordinary_differential__asymptotic__equations}
Let $\delta\in\interval0\infty$, $\ve\in\interval*0\infty$, $t_0\in\R$, $P\in\R$ be arbitrary constants and $h\in\Lp^1(\interval*{t_0}\infty)$ be an arbitrary integrable function. If a differentiable function $f:\interval{t_0}\infty\to\R$ satisfies
\[ \vert f'(t) + \frac{\delta\,f(t)}t - \frac \eta t \vert  \le \frac{h(t)}{t^\ve}, \]
then there is a number $\eta'\in\R$ {\normalfont(}depending on $f${\normalfont)} such that
\[ \left\{\quad\begin{aligned} \vert f(t) - \frac\eta\delta - \frac{\eta'}{t^\delta} \vert \le{}& \frac 1{t^\ve} \int_t^\infty h(s) \d s &:\ \ve\ge\delta, \\
									\vert f(t) - \frac\eta\delta \vert \le{}& \frac 1{t^\ve} \int_t^\infty h(s) \d s &:\ \ve\le\delta. \end{aligned}\right. \]
\end{proposition}
\begin{proof}
Comparing with $F(t):=f(t)-\nicefrac\eta\delta$, we can without loss of generality assume that $\eta=0$.

First, we assume $\ve\ge\delta$. Integration with $t_0\le t<T$ implies
\[ \vert g(t) - g(T)\vert \le \int_t^T\frac{h(s)}{s^{\ve-\delta}} \d s \le \frac1{t^{\ve-\delta}} \int_t^T h(s) \d s \le \frac1{t^{\ve-\delta}} \int_t^\infty h(s) \d s. \]
Thus, the limit $g_\infty:=\lim_{t\to\infty}g(t)$ exists and satisfies
\[ \vert g_\infty -g(t) \vert \le \int_t^\infty\frac{h(s)}{s^{\ve-\delta}} \d s \]
which implies the claim for $\ve\ge\delta$.\smallskip

Now, let us assume $\ve\le\delta$. We know
\[ \vert\partial[t]@{|f(t)|} + \frac{\delta\,|f(t)|}t\vert \le \frac{h(t)}{t^\ve} \]
for almost every $t\ge t_0$ and therefore conclude
\[ \big||f(t)|-|f(T)|\big| \le \int_t^T {-}\frac\delta s\,\vert f(s)\vert + \frac{h(s)}{s^\ve} \d s \le t^{-\ve}\int_t^\infty h(s) \d s. \]
In particular, the limit $|f|_\infty:=\lim_{t\to\infty}|f(t)|$ exists and satisfies
\[ \big||f|_\infty - |f(t)|\big| \le t^{-\ve}\int_T^\infty h(s) \d s. \]
Thus, we only have to prove $|f|_\infty=0$ and therefore assume $|f|_\infty>0$. As $f$ is continuous, this implies that (without loss of generality) $f(t)>0$ for every $t>T$ for some sufficiently large $T$. In particular, this implies
$\vert f_\infty - f(t) \vert \le t^{-\ve}\int_t^\infty h(s)\d s$
for $f_\infty:=|f|_\infty$. Again calculating the derivative, we see
\begin{align*}
 \vert\partial[t]@{f(t)} + \frac\delta t\,f_\infty\vert
	\le{}& \vert\partial[t]@{f(t)} + \frac\delta t\,f(t)\vert + \delta t^{{-}1-\ve}\,\int_t^\infty h(s) \ds \\
	\le{}& t^{-\ve}(h(t) + \frac\delta t\,\int_t^\infty h(s) \ds).
\end{align*}
Hence, we conclude by integration that
\begin{align*}
 \infty >{}& \vert f_\infty - f(t)\vert
	\ge \int_t^\infty \delta\,f_\infty\,s^{-1} - s^{-\ve}(h(s) + \frac\delta s\,\int_s^\infty h(u) \d u) \ds \\
	\ge{}& \delta\,f_\infty\, \int_t^\infty s^{-1} \d s - c = \infty
\end{align*}
holds for some $c<\infty$. By this contradiction, we get $f_\infty=0$ proving the claim.
\end{proof}
\bibliography{bib}
\makeatletter%
\def\bibindent{10em}%
\let\old@biblabel\@biblabel%
\def\@biblabel#1{\old@biblabel{#1}\kern\bibindent}%
\makeatother%
\bibliographystyle{alpha}\vfill%
\end{document}